\documentclass{article}
\usepackage{graphicx} % Required for inserting images

\usepackage{setspace}

\usepackage{amsmath}
\usepackage{amssymb}
\usepackage{amsthm}
\usepackage{mathtools}
\usepackage{empheq}
\usepackage{enumitem}
\usepackage{titlesec}
\usepackage{graphicx}
\usepackage{caption}
\usepackage{subcaption}
\usepackage{diagbox}
\usepackage{hyperref}

\usepackage[T1]{fontenc}
\usepackage[utf8]{inputenc}
\usepackage{comment}
\usepackage{indentfirst}
\usepackage[top=1.25in,left=1.25in,right=1.25in]{geometry}

\numberwithin{equation}{section}

\title{\Large{\uppercase{\bf Common neighbours in planar graphs}}}
\author{\Large{Riccardo W. Maffucci}}
\date{}
\newcommand{\Addresses}{  
		R.W.~Maffucci, \textsc{Dipartimento di Matematica, Universit\`a di Torino\\\indent Via Carlo Alberto 10, Turin 10123, Italy}\par\nopagebreak\vspace{-0.35cm}
		\textit{E-mail address}, R.W.~Maffucci: \href{mailto:riccardowm@hotmail.com}{\texttt{riccardowm@hotmail.com}}
  }

\def\cc{\mathcal{C}}
\def\ce{\mathcal{E}}
\def\cq{\mathcal{Q}}
\def\calr{\mathcal{R}}
\def\ct{\mathcal{T}}

\def\fa{\mathfrak{A}}
\def\ft{\mathfrak{T}}

\newtheorem{thm}{Theorem}[section]
\newtheorem{lemma}[thm]{Lemma}
\newtheorem{prop}[thm]{Proposition}
\newtheorem{cor}[thm]{Corollary}
\newtheorem{defin}[thm]{Definition}

\begin{document}
\titleformat{\section}
  {\Large\scshape}{\thesection}{1em}{}
\titleformat{\subsection}
  {\large\scshape}{\thesubsection}{1em}{}
\maketitle
\Addresses

\begin{abstract}
For every positive integer $n$, we find a complete classification for planar graphs according to the collection of numbers of common neighbours for every $n$-tuple of distinct vertices. Our results expand the literature on planar graphical degree sequences, that have recently been the object of renewed attention. Here we completely settle the version with no multiplicities of the vast problem of planar graphical $n$-degree sequences.
\end{abstract}

%\tableofcontents

\section{Introduction}
\subsection{Common neighbours}

Everywhere $G=(V,E)$ is a simple, finite graph on $p=|V|$ vertices and $q=|E|$ edges. The neighbourhood of $u\in V(G)$ is the set of vertices adjacent to $u$,
\[N(u)=N_G(u):=\{v\in V(G) : vu\in E(G)\}\]
(throughout we may suppress dependencies on $G$ when there is no risk of ambiguity). Its cardinality $|N(u)|$ is the degree of $u$. The maximal vertex degree will be denoted by $\Delta=\Delta(G)$.

The {\em common neighbourhood} of $n$ distinct vertices $u_1,u_2,\dots,u_n$ is the set of vertices adjacent to all of them,
\[N(u_1,u_2,\dots,u_n):=\{v\in V(G) : vu_1,vu_2,\dots,vu_n\in E(G)\}.\]

\begin{defin}
	Let $G$ be a graph and $n$ a positive integer. We define
	\begin{equation}
		\label{eq:A}
		A_n=A_n(G):=\{a : \exists \text{ distinct } u_1,u_2,\dots,u_n\in V(G) \text{ satisfying } |N(u_1,u_2,\dots,u_n)|=a\}.
	\end{equation}
\end{defin}

Note that $A_1(G)$ is simply the collection of distinct vertex degrees in $G$. The case $n=2$ was introduced by the author in \cite{maffucci2025classification}, by classifying the planar, $3$-connected (i.e., polyhedral) graphs according to their set $A_2$.

In this paper, for every positive integer $n$, we classify all planar graphs according to their set $A_n$. We also classify the subclasses of polyhedra and outerplanar graphs according to $A_n$.% (as mentioned, for $n=2$ the classification of polyhedra was established in \cite{maffucci2025classification}).

The planar, $3$-connected graphs are called polyhedra or $3$-polytopes since they are the $1$-skeleta of polyhedral solids. Further motivation for studying this intriguing class of graphs may be found in \cite[Introduction]{maffucci2025regularity}. Outerplanar graphs are graphs admitting a planar embedding where one region contains all vertices (for instance, all forests and all cycles are outerplanar). A (non-trivial) outerplanar graph always has at least two vertices of degree at most $2$ and in particular its connectivity is at most $2$ \cite[Corollary 11.9(a)]{harary}. We will focus also on this subclass since it is of general interest, and it arises naturally in one of our constructions (Lemma \ref{le:vast} to follow).

When classifying planar graphs according to $A_n$, we may restrict our attention to planar, connected graphs. Indeed, if $G_1,G_2$ are planar, connected graphs and $G=G_1\cup G_2$ (i.e., $G$ is the disjoint union of $G_1,G_2$), then
\[A_1(G)=A_1(G_1)\cup A_1(G_2)\]
and for every $n\geq 2$ we have
\[A_n(G)=\{0\}\cup A_n(G_1)\cup A_n(G_2).\]

In the following two sections, we will state our findings on the classification of planar graphs according to $A_n$, for $n\geq 3$ and for $n=2$ respectively. The case $n=1$ is relatively straightforward and is relegated to Appendix \ref{app:c}.

\subsection{Common neighbours of three or more vertices}

We denote by $P_m$ the simple path on $m\geq 2$ vertices. If $G_1,G_2$ are graphs, we denote by $G_1+G_2$ the operation of graph addition, defined by
\[V(G_1+G_2)=V(G_1)\cup V(G_2), \qquad E(G_1+G_2)=E(G_1)\cup E(G_2)\cup\{u_1u_2 : u_1\in V(G_1),\ u_2\in V(G_2)\}.\]
The notations $K_\ell$, and $K_{\ell_1,\ell_2}$ denote the complete graphs and complete bipartite graphs respectively.

Let
\begin{equation}
	\label{eq:tm}
\ct_m:=\{H+K_2 : H\text{ is a subgraph of }P_m\text{ with no isolated vertices}\}
\end{equation}
and
\[\ct:=\bigcup_{m\geq 2}\ct_m.\]
For instance, we have the family of polyhedra
\[T_m:=P_{m}+K_2,\qquad m\geq 2\]
(Figure \ref{fig:c1}). We note that $\ct_2=\{T_2\}$ where $T_2=K_4$ is the tetrahedron, and $\ct_3=\{T_3\}$ where $T_3$ is the triangular bipyramid.
\begin{figure}[ht]
	\centering
	\begin{subfigure}{0.32\textwidth}
		\centering
		\includegraphics[width=2.75cm]{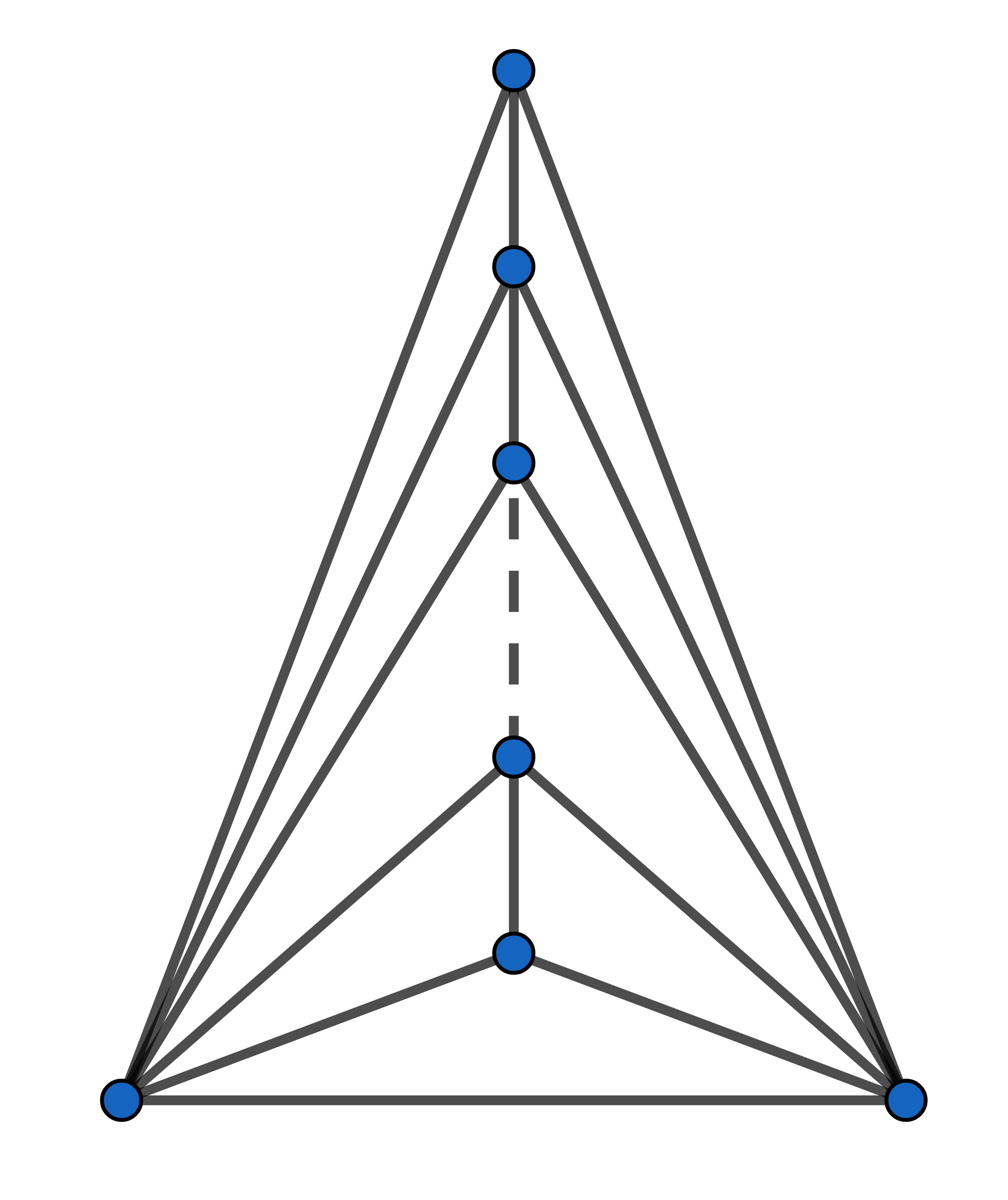}
		\caption{$T_m$, $m\geq 2$.}
		\label{fig:c1}
	\end{subfigure}
	\hfill
	\begin{subfigure}{0.32\textwidth}
		\centering
		\includegraphics[width=3.cm]{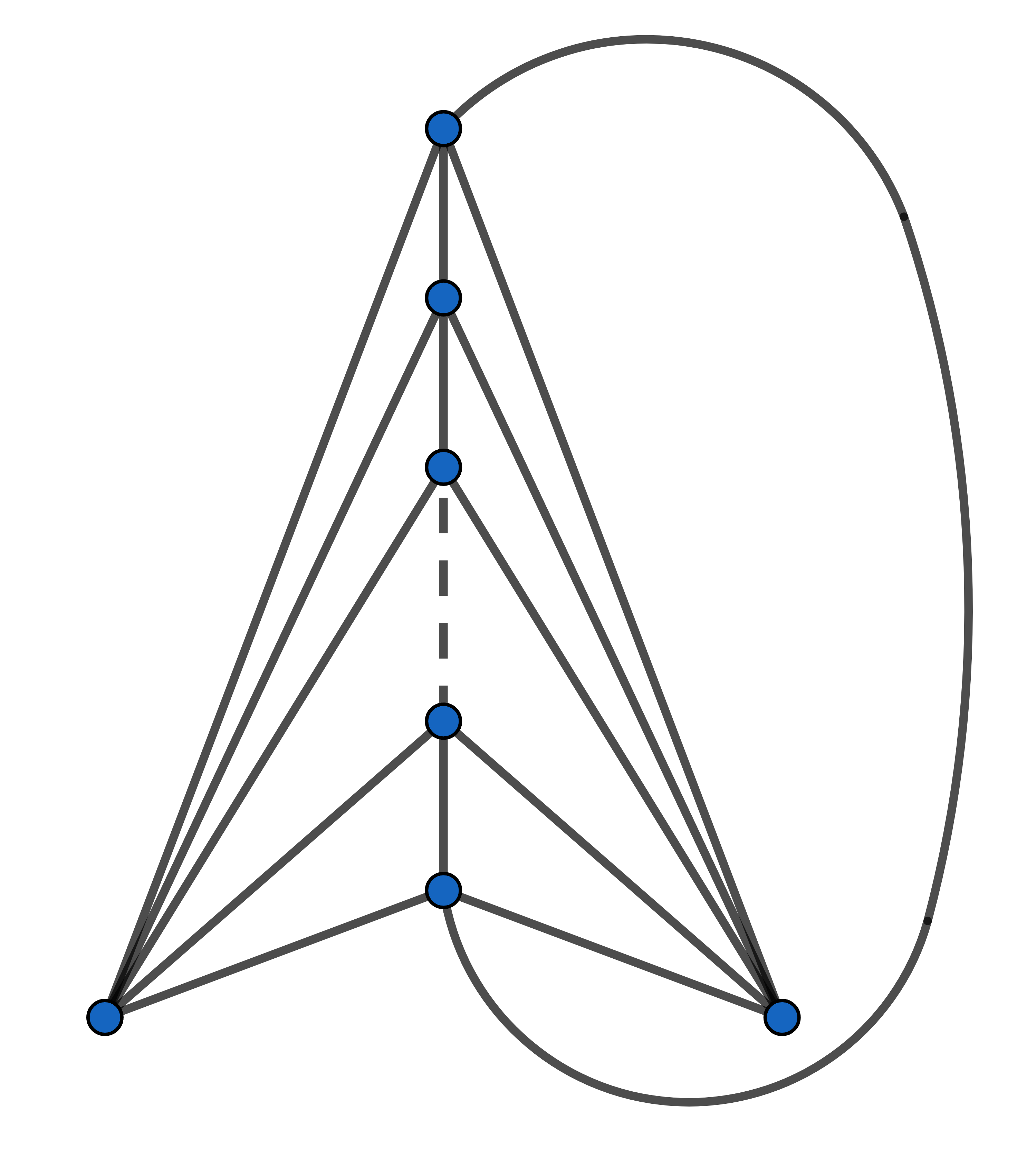}
		\caption{$B_\ell$, $\ell\geq 3$.}
		\label{fig:c2}
	\end{subfigure}
	\hfill
	\begin{subfigure}{0.32\textwidth}
		\centering
		\includegraphics[width=2.75cm]{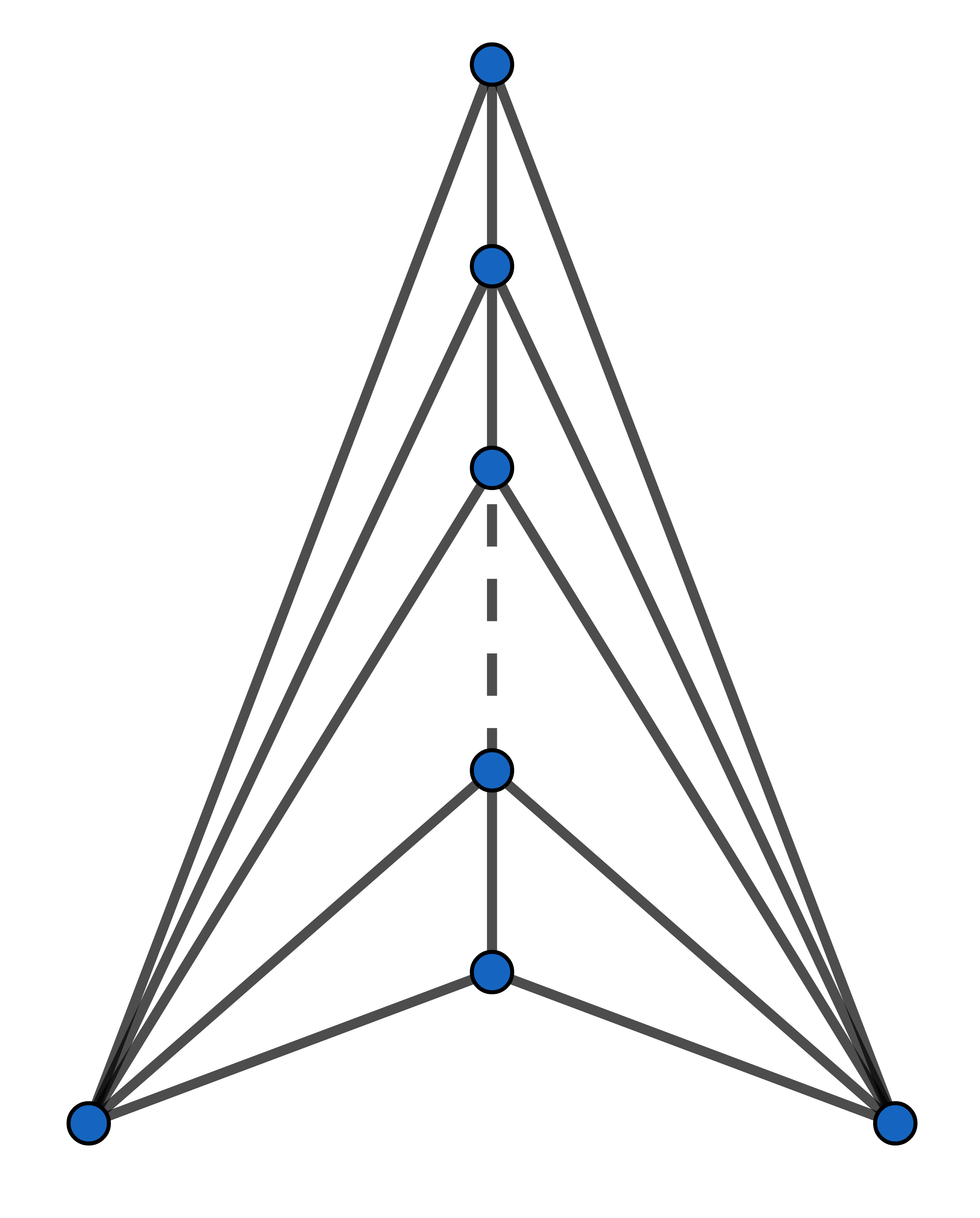}
		\caption{$B'_\ell$, $\ell\geq 3$.}
		\label{fig:c3}
	\end{subfigure}
	\caption{The families of polyhedra $T_m,B_\ell,B'_\ell$.}
	\label{fig:c}
\end{figure}

For every $m\geq 3$, let
\begin{align}
	\label{eq:qm}
\cq_m:=\{&\text{Planar graphs where }m\leq\Delta\text{ is the minimum integer s.t.\ for every vertex }v\text{ of degree at}\notag\\&\text{least }m\text{ there exists another vertex }w\text{ with same neighbourhood as }v\},
\end{align}
and
\[\cq:=\bigcup_{m\geq 3}\cq_m.\]
Note that a member of $\cq_m$ need not have any vertex of degree exactly $m$.

As an illustration of $\cq$, we will now construct two families of polyhedral graphs that will also be of use later. The notation $\overline{G}$ means the complement of the graph $G$. For $\ell\geq 3$ we define
\[B_\ell:=C_{\ell}+\overline{K_2}\quad
(\text{bipyramids}),\qquad\text{and}\\\qquad B'_\ell:=P_{\ell}+\overline{K_2}
\]
(Figures \ref{fig:c2} and \ref{fig:c3}). The octahedron $B_4$ satisfies $B_4\in\cq_3$, while $B_\ell,B'_\ell\in\cq_5$ for every $\ell\geq 5$. Note that the triangular bipyramid $B_3=T_3$, the square pyramid $B'_3$, and the graph $B'_4$ are not members of $\cq$. Further details will follow in Appendix \ref{app:a}.

We may now define an exceptional class of graphs
\begin{equation}
	\label{eq:ce}
	\ce:=\ct\cup\cq\cup\{S_5,S_7\},
\end{equation}
with $S_5,S_7$ as in Figure \ref{fig:s} (maintaining the notation of \cite[Figure 2]{maffucci2025classification} for the readers' convenience).
\begin{figure}[ht]
	\centering
	\begin{subfigure}{0.48\textwidth}
		\centering
		\includegraphics[width=2.75cm]{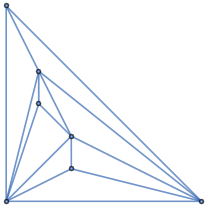}
		\caption{$S_5$.}
		\label{fig:s5}
	\end{subfigure}
	\hfill
	\begin{subfigure}{0.48\textwidth}
		\centering
		\includegraphics[width=3.cm]{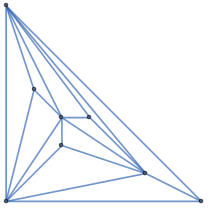}
		\caption{$S_7$.}
		\label{fig:s7}
	\end{subfigure}
	\caption{The exceptional graphs $S_5,S_7$.}
	\label{fig:s}
\end{figure}

We are ready to state our first main result.
\begin{thm}
	\label{thm:1}
	Let $G$ be a planar graph and $n\geq 3$ an integer. Then
	\begin{equation*}
		A_n(G)=
		\begin{cases}
			\{1\} & n=3 \text{ and } G\simeq K_4,\\
			\{1,2\} & n=3 \text{ and } G\in\ct_m\cup\{S_5,S_7\},\ m\geq n,\\
			\{0,2\} & G\in \cq_m,\ m\leq n,\\
			\emptyset & |V(G)|<n,\\
			\{0\} & \text{else if } \Delta(G)<n,\\
			\{0,1,2\} & \text{else if } G \text{ contains } K_{2,n},\\
			\{0,1\} & \text{otherwise}.
		\end{cases}
	\end{equation*}
	
\end{thm}

Recall that $p$ is the order of $G$ and $\Delta$ the maximal vertex degree. We define
\begin{equation}
	\label{eq:L}
	L=L(G):=\max\{\ell\in\mathbb{N} : G \text{ contains } K_{2,\ell}\}.
\end{equation}
Note that every graph satisfies
\[L\leq\Delta<p.\]

Defining
\begin{equation}
	\label{eq:fa}
	\fa=\fa(G):=\{A_3(G),A_4(G),\dots,A_n(G),\dots\}
\end{equation}
we may rephrase Theorem \ref{thm:1} more explicitly as follows.

\begin{thm}
	\label{thm:2}
	Let $G\not\in\ce$ be a planar graph. Then
	\begin{equation}
		\label{eq:Agen}
		\fa(G)=\{\{0,1,2\},\{0,1,2\},\dots,\underbrace{\{0,1,2\}}_{A_L},\{0,1\},\{0,1\},\dots,\{0,1\},\underbrace{\{0\}}_{A_{\Delta+1}},\{0\},\dots,\underbrace{\{0\}}_{A_p},\emptyset,\emptyset,\dots,\emptyset,\dots\}.
	\end{equation}
	For the exceptional class $\ce$ one has the following. For $G\in\ct_m$, $m\geq 3$,
	\begin{equation}
		\label{eq:AT}
		\fa(G)=\{\underbrace{\{1,2\}}_{A_{3}},\{0,1,2\},\{0,1,2\},\dots,\{0,1,2\},\underbrace{\{0,1\}}_{A_{m+1}},\{0\},\emptyset,\emptyset,\dots,\emptyset,\dots\}.
	\end{equation}
	For $G\in\cq_m$, $m\geq 3$,
	\begin{equation}
		\label{eq:AQ}	
		\fa(G)=\{\{0,1,2\},\{0,1,2\},\dots,\{0,1,2\},\underbrace{\{0,2\}}_{A_m},\{0,2\},\dots,\underbrace{\{0,2\}}_{A_\Delta},\{0\},\{0\},\dots,\underbrace{\{0\}}_{A_p},\emptyset,\emptyset,\dots,\emptyset,\dots\}.
	\end{equation}
	For the remaining cases,
	\begin{align}
		\label{eq:Aexc}
		\fa(K_4)&=\{\{1\},\{0\},\emptyset,\emptyset,\dots,\emptyset,\dots\},\notag\\
		%\fa(T_3)&=\{\{1,2\},\{0,1\},\{0\},\emptyset,\emptyset,\dots,\emptyset,\dots\},\notag\\
		%\fa(B_4)&=\{\{0,2\},\{0,2\},\{0\},\{0\},\emptyset,\emptyset,\dots,\emptyset,\dots\},\notag\\
		%\fa(B'_4)&=\{\{0,1,2\},\{0,1,2\},\{0\},\{0\},\emptyset,\emptyset,\dots,\emptyset,\dots\},\notag\\
		\fa(S_5)&=\{\{1,2\},\{0,1,2\},\{0,1\},\{0,1\},\{0\},\emptyset,\emptyset,\dots,\emptyset,\dots\}\notag\\
		\fa(S_7)&=\{\{1,2\},\{0,1,2\},\{0,1\},\{0,1\},\{0\},\{0\},\emptyset,\emptyset,\dots,\emptyset,\dots\}.
	\end{align}
\end{thm}
Theorems \ref{thm:1} and \ref{thm:2} will be proven in Section \ref{sec:2}.

We specify that in \eqref{eq:Agen} if $L\leq 2$, then $\{0,1,2\}\not\in\fa$, and if $\Delta\leq 2$, then $\{0,1\}\not\in\fa$. Similarly, in \eqref{eq:AT} and \eqref{eq:AQ} if $m=3$, then $\{0,1,2\}\not\in\fa$.

\subsection{Common neighbours of a pair of vertices}

We denote by
\begin{equation}
	\label{eq:D}
D_\ell,\qquad\ell\geq 1
\end{equation}
the graph formed by $\ell$ triangles with a common vertex. We also define
\begin{equation}
	\label{eq:Tp}
T'_\ell \simeq \left(\bigcup_{i=1}^{\ell/2}K_2\right)+K_2\in\ct_\ell,\qquad\text{even }\ell\geq 2.
\end{equation}
Note that $T'_2=T_2$ and $\ct_4=\{T_4,T'_4\}$.

In Figure \ref{fig:ee} we sketch five polyhedra from \cite[Figure 2]{maffucci2025classification} and a planar, $2$-connected graph $S'$.
\begin{figure}[ht]
	\centering
	\begin{subfigure}{0.32\textwidth}
		\centering
		\includegraphics[width=2.25cm]{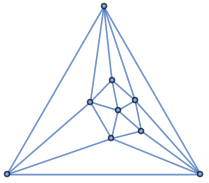}
		\caption{$S_3$.}
		\label{fig:s3}
	\end{subfigure}
	\hfill
	\begin{subfigure}{0.32\textwidth}
		\centering
		\includegraphics[width=2.25cm]{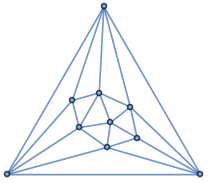}
		\caption{$S_4$.}
		\label{fig:s4}
	\end{subfigure}
	\hfill
	\begin{subfigure}{0.32\textwidth}
		\centering
		\includegraphics[width=2.25cm]{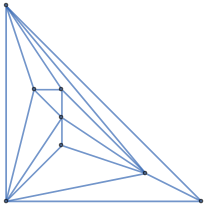}
		\caption{$S_6$.}
		\label{fig:s6}
	\end{subfigure}
	\hfill
	\begin{subfigure}{0.32\textwidth}
		\centering
		\includegraphics[width=2.25cm]{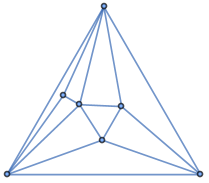}
		\caption{$S_8$.}
		\label{fig:s8}
	\end{subfigure}
	\hfill
	\begin{subfigure}{0.32\textwidth}
		\centering
		\includegraphics[width=2.25cm]{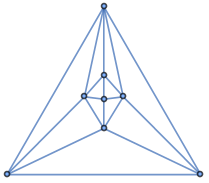}
		\caption{$S_9$.}
		\label{fig:s9}
	\end{subfigure}
	\hfill
	\begin{subfigure}{0.32\textwidth}
		\centering
		\includegraphics[width=2.25cm]{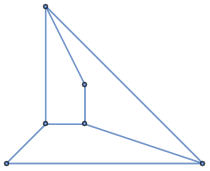}
		\caption{$S'$.}
		\label{fig:s11}
	\end{subfigure}
	\caption{The polyhedra $S_3,S_4,S_6,S_8,S_9$ and the planar, $2$-connected graph $S'$.}
	\label{fig:ee}
\end{figure}

We may now define an exceptional class of graphs
\[
%\label{eq:cep}
%\ce':=\{D_\ell\}_{\ell\geq 1}\cup\{T_\ell\}_{\ell\geq 2}\cup\{B_\ell\}_{\ell\geq 3}\cup\{S_\ell\}_{3\leq\ell\leq 9}\cup\{K_1,K_2\}.
\ce':=\ct\cup\{B_\ell\}_{\ell\geq 3}\cup\{S_\ell\}_{3\leq\ell\leq 9}.
\]

\begin{thm}
	\label{thm:3}
Let $G\not\in\ce'$ be a planar graph containing a $4$-cycle. Then $A_2(G)$ contains $2$ and at least one of $0,1$. Conversely, let $A'$ be a (possibly empty) set of integers $\geq 3$. Then we may construct a vast class of planar, connected graphs satisfying
\[A_2(G)=\{1,2\}\cup A',\]
a vast class of planar, connected graphs satisfying
\[A_2(G)=\{0,1,2\}\cup A',\]
and for $A'\neq\emptyset,\{3\}$, a vast class of planar, connected graphs satisfying
\[A_2(G)=\{0,2\}\cup A'.\]
The full classification of planar, connected graphs $G$ according to $A_2(G)$ may be found in Table \ref{t:1}.
\begin{table}[ht]
	\centering
	$\begin{array}{|c|c|}
		\hline A_2(G)&G\\
		\hline \emptyset&K_1\\
		\hline \{0\}&K_2\\
		\hline \{2\}&T_2=K_4\\
		\hline \{2,3\}&T_3=B_3, S_3\\
		\hline \{2,4\}&B_4,T'_4\\
		\hline \{2,\ell\},\text{ even }\ell\geq 6&T'_\ell\\
		\hline \{2,3,4\}&T_4,S_5,S_6,S_7,S_8,S_{9}\\
		\hline \{2,3,\ell\}, \ \ell\geq 5&G\simeq B_\ell\ \text{ or }\ G\in\ct_{\ell},\ G\not\simeq T'_\ell\\
		\hline \{0,2\}&\text{cube, icosahedron, square}\\
		\hline \{0,2,3\}&S_4,K_{2,3},S'\\
		%\hline \{0,2,3,4\}&S_{10}\\
		\hline \{0,2\}\cup A',\ A'\neq\emptyset,\{3\}&\text{infinitely many solutions for each }A'\\
		\hline \{1\}&D_\ell,\ \ell\geq 1\\
		\hline \{0,1\}&\text{no $4$-cycles},\ G\not\simeq D_\ell,K_1,K_2\\
		\hline \{1,2\}\cup A'&\text{infinitely many solutions for each }A'\\
		\hline \{0,1,2\}\cup A'&\text{infinitely many solutions for each }A'\\
		\hline
	\end{array}$
	\caption{Classification of planar, connected graphs $G$ according to $A_2(G)$.}
	\label{t:1}
\end{table}

\end{thm}

Theorem \ref{thm:3} will be proven in Section \ref{sec:3}.

\subsection{Discussion}
\label{sec:disc}

The problem of classifying planar graphs according to $A_1(G)$ is the version without multiplicities of the classical problem of determining which degree sequences are planar graphical. Given a graph $G$, we may consider its {\em degree sequence}
\[\sigma: d_1,d_2,\dots,d_p,\]
collecting the vertex degrees of $G$ in weakly decreasing order. We say that $G$ is a realisation of $\sigma$. The classical problem of determining when a sequence is {\em graphical} (has at least one realisation) was solved in \cite{have55,hakimi,erdgal}. The authors of \cite{have55,hakimi} also established an algorithm to construct a realisation for every graphical sequence.

The difficult question of which sequences are {\em planar graphical} (i.e., admit a planar realisation) was posed in \cite{schhak}. This theory has recently received renewed attention \cite{adams2019planar,maffucci2024rao,bar2024sparse,bar2025approximate}. The problem of constructing planar graphs $G$ given $A_1(G)$ is the version without multiplicities of determining which degree sequences are planar graphical. Its solution may be found in Appendix \ref{app:c}.

We now generalise both the problem of planar graphical sequences, and the problem analysed in this paper of classifying planar graphs according to $A_n$, $n\geq 1$. We start by writing
\[V^{(n)}(G):=\{(u_1,u_2,\dots,u_n) : \text{ distinct }u_1,u_2,\dots,u_n\in V(G)\}.\]
%and for $W\subseteq V^{(n)}(G)$,
%\[d_W:=|N(u_1,u_2,\dots,u_{n})|.\]
For each positive integer $n$, we may call $n$-degree sequence of $G$ the list of non-negative integers
\begin{equation}
\label{eq:nds}
|N(u_1,u_2,\dots,u_{n})|, \qquad (u_1,u_2,\dots,u_{n})\in V^{(n)}(G),
\end{equation}
ordered such that  \eqref{eq:nds} is weakly decreasing. The case $n=1$ is the usual degree sequence of $G$. The subset $A_n$ is \eqref{eq:nds} without multiplicities. The problems of determining which $n$-degree sequences admit a realisation, whether there are few or many realisations, and how to construct them seem intriguing and non-trivial.

The related problem of {\em unigraphicity} \cite{koren1,li1975} is also of interest. We say that a sequence $\sigma$ is unigraphic with respect to a class of graphs $\cc$ if there is a unique $G\in\cc$ realising $\sigma$. In \cite{maffucci2024characterising}, the difficult problem of finding unigraphic sequences with respect to the class of polyhedra was posed, and the case of polyhedra with a dominating vertex was partially resolved. Equivalently, \cite{maffucci2024characterising} partially solves the problem of outerplanar, $2$-connected unigraphic sequences (this class may be obtained from the above by removing the dominating vertex). In \cite{maffucci2025faces}, we achieved the breakthrough of proving that unigraphic polyhedra may have face length at most $9$. The case of self-dual unigraphic polyhedra was solved in \cite{maffucci2026self}. Other problems on degree sequences of polyhedral graphs were studied in \cite{mafpo2,maffucci2023self}.

\paragraph{Preliminaries.}
A vertex of a graph is dominating (sometimes also called universal) if it is adjacent to every other vertex.
\\
A caterpillar is a tree graphs where if we remove all vertices of degree $1$, we are left with a simple path.
\\
Let $u\in V(G)$, and $N(u)=\{v_1,v_2,\dots,v_n\}$. The {\em open neighbourhood} of $u$
\[\Gamma_G(u)\]
is the subgraph of $G$ generated by $v_1,v_2,\dots,v_n$. The {\em closed neighbourhood} of $u$
\[\Gamma_G[u]\]
is the subgraph of $G$ generated by $u,v_1,v_2,\dots,v_n$.
\\
A planar graph considered together with a planar embedding is called a plane graph. Polyhedra admit a unique embedding up to selecting an external region \cite[Introduction]{maffucci2025regularity}, and this fact will be tacitly used throughout.
\\
In a plane graph $G$, we will use the notation
\[[u_1,u_2,\dots,u_n], \qquad n\geq 3\]
when $u_1,u_2,\dots,u_n$ is a cycle that delimits a region of $G$.

\paragraph{Plan of the paper.}
Section \ref{sec:2} is dedicated to the proofs of Theorems \ref{thm:1} and \ref{thm:2}. Section \ref{sec:3} is dedicated to the proof of Theorem \ref{thm:3}. In Appendix \ref{app:a}, we will classify all polyhedral graphs according to their set $A_n$, for every $n\geq 3$. In Appendix \ref{app:b}, we will classify all outerplanar graphs according to their set $A_n$, for every $n\geq 2$. In Appendix \ref{app:c}, we will classify all planar graphs according to their set $A_1$.

\section{Proof of Theorems \ref{thm:1} and \ref{thm:2}}
\label{sec:2}
Let $n\geq 3$ and $G$ be a planar graph. Since $G$ is planar, a subset of vertices of cardinality $3$ or more may have at most two common neighbours, otherwise $G$ would contain a subgraph isomorphic to $K_{3,3}$. Thus
\[A_n(G)\subseteq\{0,1,2\}.\]

\subsection{The exceptional class $\ct$}
One easily checks that the graphs $T_2=K_4,S_5,S_7$ satisfy \eqref{eq:Aexc}, and that the triangular bipyramid $T_3$ satisfies \eqref{eq:AT}. Now we will consider $G\in\ct_m$, $m\geq 4$ and show that it verifies \eqref{eq:AT}.

\paragraph{Case of $\ct_m$.} Let $G\in\ct_m$, $m\geq 4$. That is to say, $G=H+K_2$ where $H$ is a subgraph of the path $P_m$ with no isolated vertices. Call
\[w_1,w_2,\dots,w_m\]
the vertices along $P_m$ in order, and $x,y$ the remaining vertices of $G$. One has
\[N(x,y,w_1,w_2,\dots,w_{i})=\emptyset, \qquad 2\leq i\leq m,\]
thus $0\in A_n$ for $4\leq n\leq m+2$,
\[N(x,w_1,w_2,\dots,w_{i})=\{y\}, \qquad 2\leq i\leq m,\]
thus $1\in A_n$ for $3\leq n\leq m+1$, and
\[N(w_1,w_2,\dots,w_{i})=\{x,y\}, \qquad 2\leq i\leq m,\]
thus $2\in A_n$ for $2\leq n\leq m$. Thereby,
\[A_n(G)=\{0,1,2\}, \qquad 4\leq n\leq m\]
and
\[A_3(G)\supseteq\{1,2\}, \qquad 4\leq m.\]
On the other hand, let us check that $0\not\in A_3$. Indeed, any three elements of $P_m$ are all adjacent to $x,y$; $x$ and any two vertices of $P_m$ are all adjacent to $y$; $x,y$ and any vertex $w$ of $P_m$ are all adjacent to the neighbour(s) of $w$ along $P_m$ (recall that $w$ is not isolated in $H$). Thus for every $m\geq 4$ we have $A_3(G)=\{1,2\}$. We have verified \eqref{eq:AT}.

\subsection{Main arguments}
It remains to prove the generic case \eqref{eq:Agen} and the exceptional \eqref{eq:AQ}. The purpose of the present section is to establish three preparatory results. These are at the heart of the proof of Theorems \ref{thm:1} and \ref{thm:2}.
\begin{lemma}
	\label{le:no0}
Let $G$ be a planar graph, and $n\geq 4$ an integer. If $A_n(G)\neq\emptyset$, then $0\in A_n(G)$.
\end{lemma}
\begin{proof}
By contradiction, assume that $0\not\in A_n$. Then every $4$-tuple of vertices has a common neighbour. Since $A_n(G)\neq\emptyset$, then $|V(G)|\geq 5$. We take
\[vu_1,vu_2,vu_3,vu_4\in E(G).\]
Given any four of $v,u_1,u_2,u_3,u_4$, they have a common neighbour. Since $G$ is planar, it does not contain a subgraph isomorphic to $K_5$, so that $v,u_1,u_2,u_3,u_4$ cannot be all pairwise adjacent. Hence w.l.o.g., there exists $w\neq u_4$ adjacent to $v,u_1,u_2,u_3$, as in Figure \ref{fig:yes0a}. Since
\[N(v,w,u_2,u_4)\neq\emptyset,\]
then by planarity we have $N(v,w,u_2,u_4)=\{u_i\}$ with $i\in\{1,3\}$ w.l.o.g, $i=1$, as in Figure \ref{fig:yes0b}. Now $N(v,w,u_1,u_2)\neq\emptyset$, but as $v,w,u_1,u_2$ are all pairwise adjacent, $G$ contains a copy of $K_5$, contradicting planarity.
\begin{figure}[ht]
	\centering
	\begin{subfigure}{0.48\textwidth}
		\centering
		\includegraphics[width=4.cm]{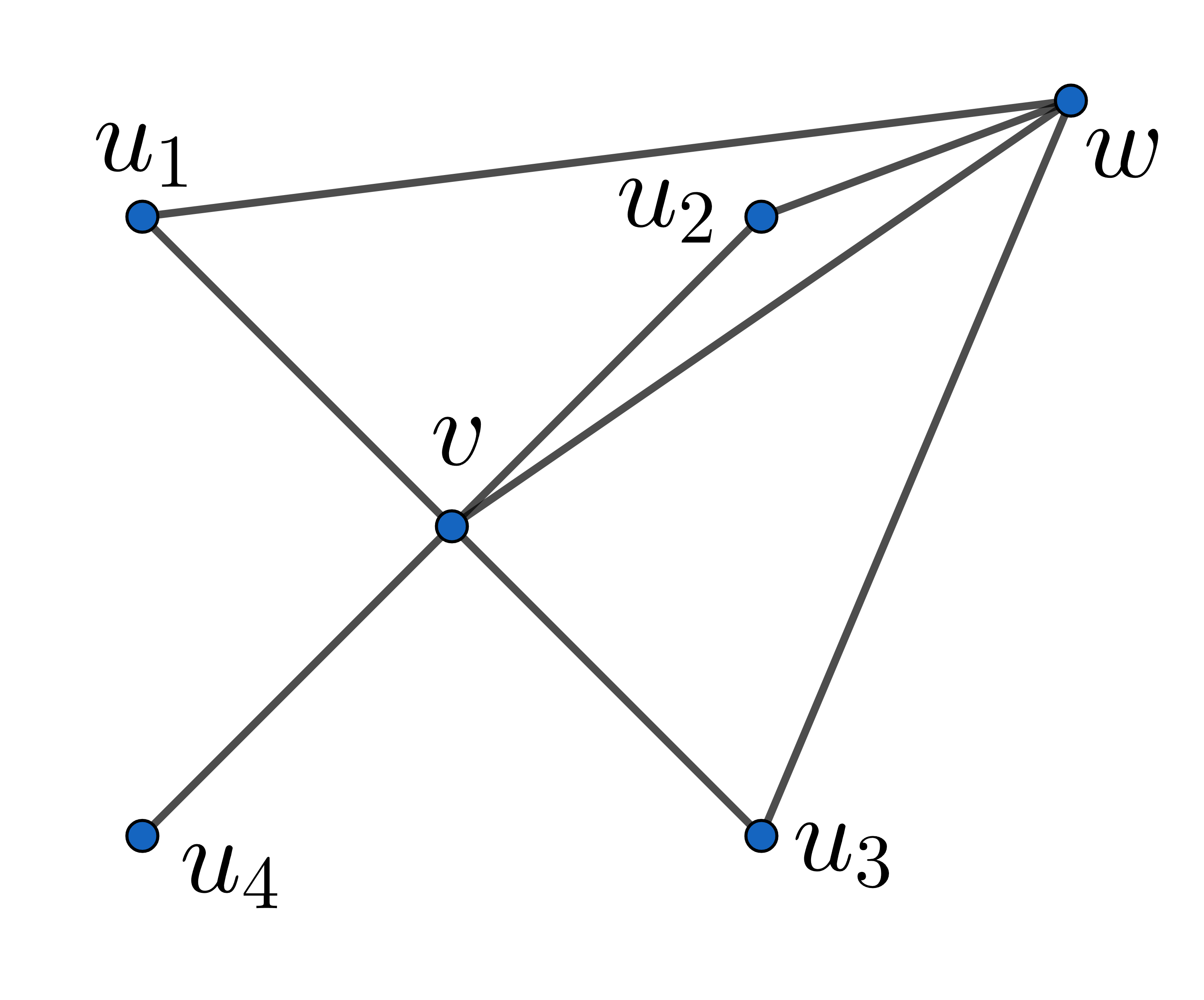}
		\caption{There exists $w\neq u_4$ adjacent to $v,u_1,u_2,u_3$.}
		\label{fig:yes0a}
	\end{subfigure}
	\hfill
	\begin{subfigure}{0.48\textwidth}
		\centering
		\includegraphics[width=4.cm]{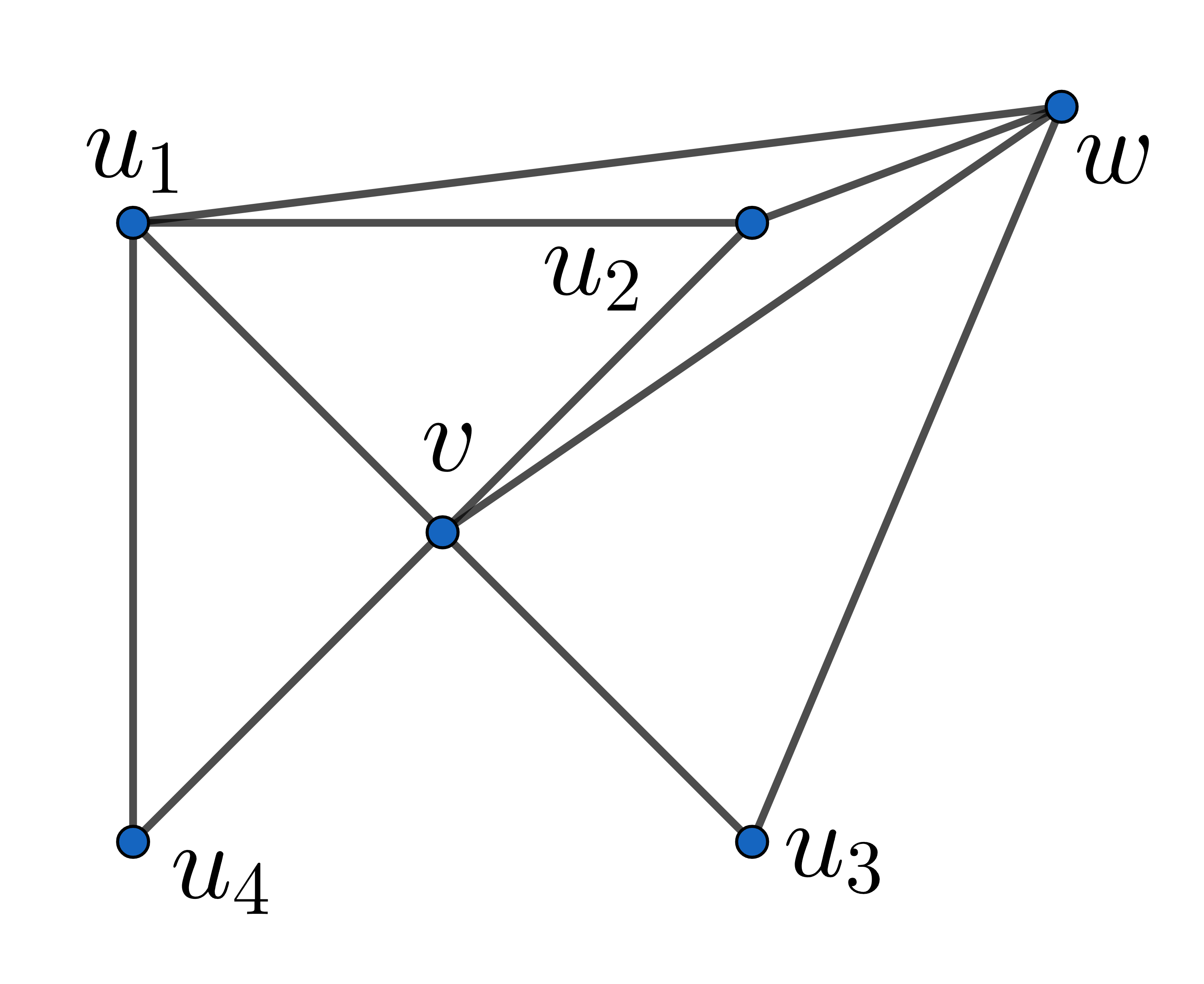}
		\caption{$N(v,w,u_2,u_4)=\{u_1\}$.}
		\label{fig:yes0b}
	\end{subfigure}
	\caption{Lemma \ref{le:no0}.}
	\label{fig:yes0}
\end{figure}
\end{proof}

\begin{prop}
	\label{prop:no03}
Let $G$ be a planar graph on at least three vertices. If $0\not\in A_3(G)$, then either $G\in\{S_5,S_7\}$ or $G\in\ct_m$ for some $m\geq 2$.
\end{prop}
\begin{proof}
Firstly, we note that if $z_0$ is a separating vertex of $G$ and $z_1,z_2\in V(G)$ are in distinct connected components of $G-z_0$, then $N(z_0,z_1,z_2)=\emptyset$. Therefore, $G$ is $2$-connected. It follows that in every planar immersion of $G$, each region is delimited by a cycle \cite[Proposition 4.2.5]{dieste}.

Next, we wish to prove that in any planar immersion of $G$, each region is delimited by a triangle or quadrangle.  By contradiction, one considers an immersion of $G$ in the plane, and a region
\[F=[u_1,u_2,\dots,u_\ell], \quad\ell\geq 5.\]
Assume by contradiction that $u_1u_i\in E(G)$, $i\neq 2,\ell$. For every $1<j<i<k\leq\ell$ one has
\[N(u_1,u_{j},u_{k})\neq\emptyset \qquad\text{and}\qquad N(u_i,u_{j},u_{k})\neq\emptyset\]
thus by planarity $u_ju_1,u_ju_i,u_ku_1,u_ku_i\in E(G)$. If $i\geq 4$, then $v_1v_3,v_2v_4\in E(G)$, while if $i=3$, then $v_1v_4,v_3v_5\in E(G)$, in either case contradicting planarity. We deduce that for every $1\leq i,j\leq\ell$ with $|i-j|\geq 2$ one has $u_iu_j\not\in E(G)$.

Now one writes
\[N(u_1,u_3,u_5)\supseteq\{v\}.\]
As shown above, $v\not\in F$. Since $N(u_2,u_{3},u_{4})\neq\emptyset$, then by planarity $v$ is adjacent to $u_1,u_2,u_3,u_4,u_5$. Now for $6\leq i\leq\ell$ we have $N(u_2,u_{4},u_{i})\neq\emptyset$ so that by planarity $vu_i\in E(G)$. It follows that $v$ is adjacent to every vertex of $F$.
%Next, we wish to prove that $G$ is a maximal planar graph. By contradiction, one considers an immersion of $G$ in the plane, and a region
%\[F=[u_1,u_2,\dots,u_\ell], %\quad\ell\geq 4.\]
%Since $G$ is $2$-connected, $u_1,u_2,\dots,u_\ell$ is a cycle. Taking the indices modulo $\ell$, one has \[N(u_i,u_{i+1},u_{i+2})\neq\emptyset, \quad 1\leq i\leq\ell,\] so that by planarity there exists $v\in V(G)$ verifying \[N(u_1,u_2,\dots,u_\ell)=\{v\}.\]

Next, there exist $w,x\in V(G)$ such that $N(v,u_1,u_2)\supseteq\{w\}$ and $N(v,u_3,u_4)\supseteq\{x\}$, as in Figure \ref{fig:no4face}. By planarity, $N(v,w,x)=\emptyset$, contradiction. Thus indeed in every immersion of $G$ every region is either triangular or quadrangular.
\begin{figure}[ht]
	\centering
	\includegraphics[width=4cm]{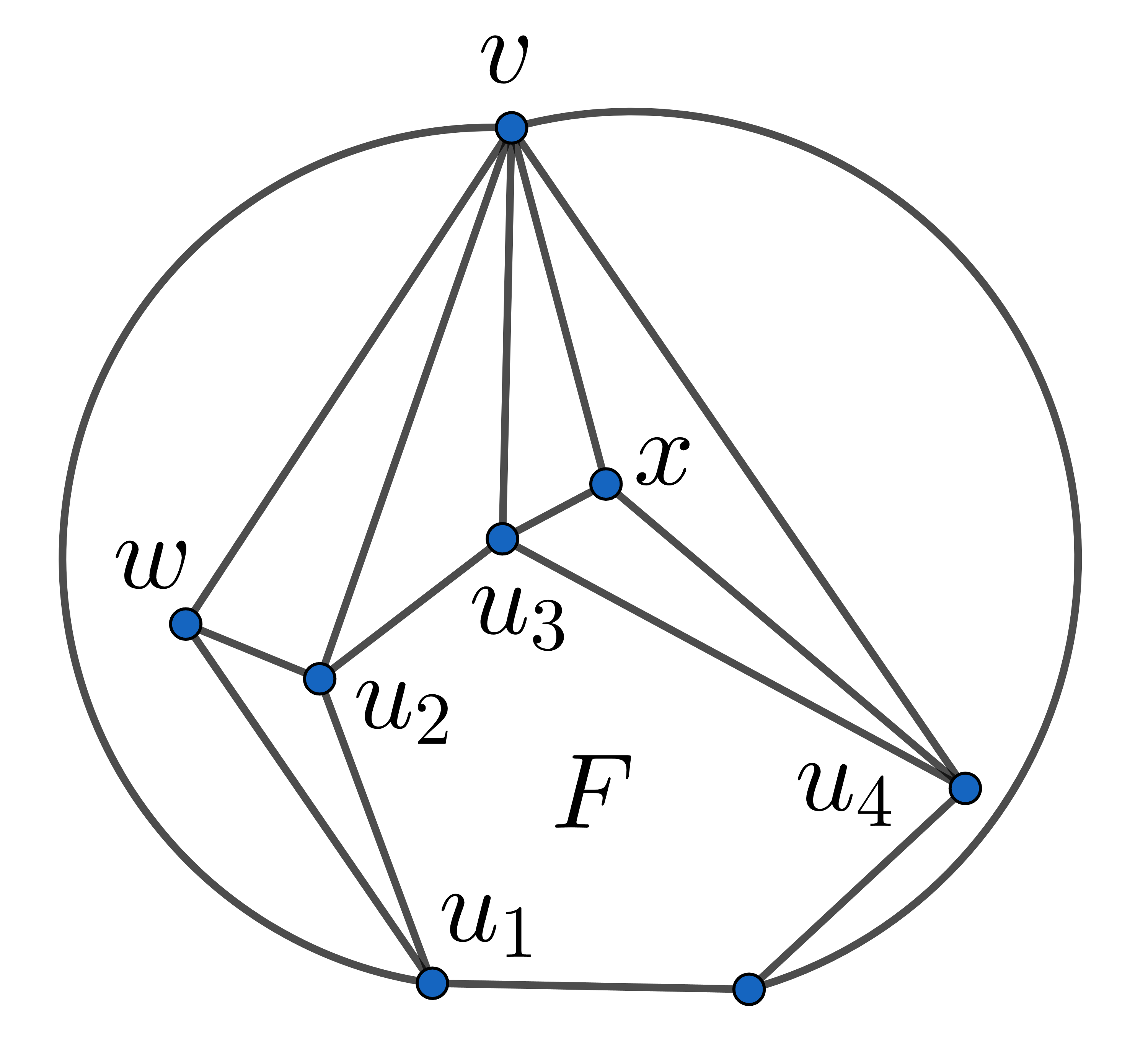}
	\caption{If $G$ is planar and $0\not\in A_3(G)$, then in every plane immersion of $G$, every region is either triangular or quadrangular.}
	\label{fig:no4face}
\end{figure}

One now considers an immersion of $G$ in the plane, and a quadrangular region
\[F=[u_1,u_2,u_3,u_4].\]
Due to planarity we cannot have both $u_1u_3,u_2u_4\in E(G)$, thus w.l.o.g.\ $u_2u_4\not\in E(G)$. One writes
\[N(u_1,u_2,u_3)\supseteq\{v\}\qquad\text{and}\qquad N(u_3,u_4,u_1)\supseteq\{v'\},\]
with $v\neq u_4$ and $v'\neq u_2$. If $v=v'$, then as in the case $\ell\geq 5$ we have $N(v,u_1,u_2)\supseteq\{w\}$ and $N(v,u_3,u_4)\supseteq\{x\}$ where $w\neq x$ are two new vertices, so that by planarity $N(v,w,x)=\emptyset$, contradiction. Hence $v\neq v'$.

Since $N(u_1,u_2,u_4)\neq\emptyset$, by planarity we have $u_1u_3\in E(G)$. Now let $w$ be a new vertex. Since $N(w,u_1,u_2)$ and $N(w,u_1,u_4)$ are both non-empty, then by planarity $wu_3\in E(G)$. Likewise, since $N(w,u_3,u_2)$ and $N(w,u_3,u_4)$ are both non-empty, then by planarity $wu_1\in E(G)$. Hence each vertex of $G$ other than $u_1,u_3$ is adjacent to both of $u_1,u_3$, and moreover, given $w\in V(G)\setminus\{u_1,u_3\}$, there exists $x\in V(G)\setminus\{u_1,u_3,w\}$ adjacent to $w$. It follows that
\[G\in\ct_m,\qquad m=\Delta(G)-1=|V(G)|-2\geq 4,\]
as claimed.

The only possibility not yet considered is that $G$ is a triangulation (maximal planar graph). Recall that a triangulation is a polyhedron. Thereby, let $G$ be a maximal planar graph such that $0\not\in A_3(G)$. Our goal is to show that $G$ is either $T_m$ for some $m\geq 2$ or one of $S_5,S_7$. We begin by considering $a,b,c\in V(G)$ forming a triangle in $G$. As $0\not\in A_3(G)$, there is a fourth vertex $d$ adjacent to $a,b,c$. These four vertices form a tetrahedron $G'$. The remaining vertices lie inside of the faces
\[[a,b,c],\ [a,b,d],\ [a,c,d],\ [b,c,d]\]
of $G'$.

Let $u_1,u_2,\dots,u_\ell$, $\ell\geq 2$ lie inside $[a,b,c]$ and $v_1,v_2,\dots,v_\ell'$, $\ell'\geq 2$ lie inside $[a,b,d]$. Since $0\not\in A_3$ and $G$ is planar, for $1\leq i\leq\ell$ and $1\leq i'\leq\ell'$ one has $N(u_i,v_{i'},a)=\{b\}$, and $N(u_i,v_{i'},b)=\{a\}$. Therefore, each of $u_1,u_2,\dots,u_\ell,v_1,v_2,\dots,v_\ell'$ is adjacent to $a$ and $b$. Assuming w.l.o.g.\ that
\begin{equation}
\label{eq:uvpath}
u_1,u_2,\dots,u_\ell,c,d,v_{\ell'},v_{\ell'-1},\dots,v_1
\end{equation}
appear in this order around $a$ in the planar immersion, then by planarity and $3$-connectivity of $G$ there is a path containing the vertices \eqref{eq:uvpath} in this order, as in Figure \ref{fig:abcd1}. Note that the graph obtained so far is $T_{\ell+\ell'+2}$. Let $w\in V(G)$ lie inside of the face $[a,c,d]$ of $G'$. Then we cannot have $\ell\geq 2$, otherwise $N(u_1,a,w)=\emptyset$.

\begin{figure}[ht]
\centering
\includegraphics[width=6.cm]{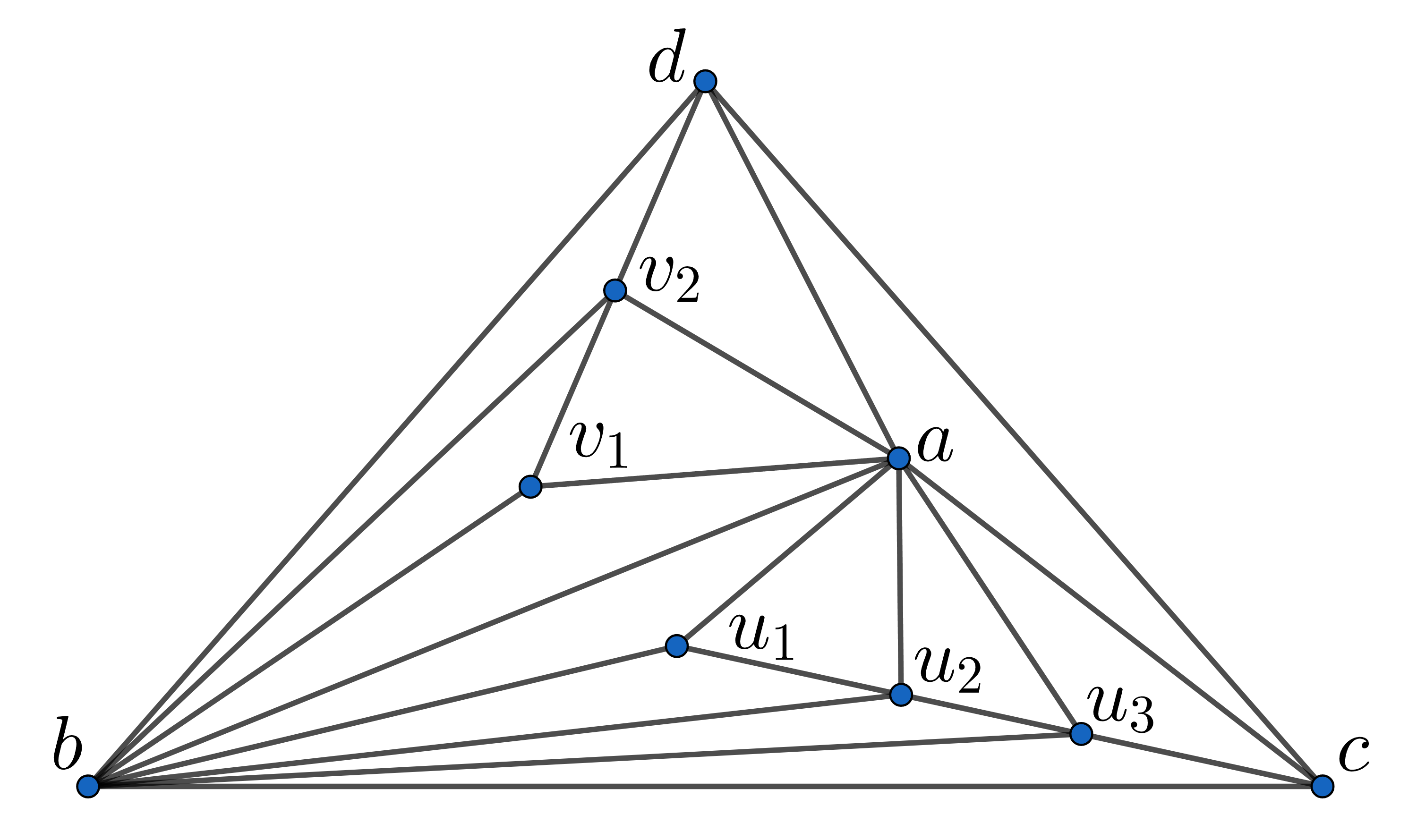}
\caption{$u_1,u_2,u_3$ lie inside $[a,b,c]$ and $v_1,v_2$ lie inside $[a,b,d]$.}
\label{fig:abcd1}
\end{figure}

We deduce that either $G=T_m$, $m\geq 2$, or inside each face of $G'$ there is at most one vertex of $G$, or all vertices of $G$ save $a,b,c,d$ lie inside the same face of $G'$, w.l.o.g.\ $[a,b,c]$. Note that by $3$-connectivity of $G$, if there is exactly one vertex of $G$ lying inside a face $F$ of $G'$, then this vertex is adjacent to all three vertices on the boundary of $F$. In case inside each face of $G'$ there is at most one vertex of $G$, then we recover the construction in \cite{maffucci2024rao}, thus $G$ is one of $T_2,T_3,T_4,S_5,S_7$ (in \cite{maffucci2024rao}, these polyhedra are the ones of degree sequences labeled respectively by $s_0,s_1,s_2,s_3,s_4$).

Now assume instead that all elements of
\[U:=V(G)\setminus\{a,b,c,d\}\]
lie inside of the face $[a,b,c]$ of $G'$. For each $z\in U$ one has $N(d,z,a),N(d,z,b),N(d,z,c)\neq\emptyset$, hence $z$ is adjacent to at least two of $a,b,c$. Conversely, since $G$ is a triangulation, there exist $u,v,w\in V(G)$ distinct from $d$ such that in $G$ we have the faces
\[[a,b,u],\ [a,c,v],\ [b,c,w]\]
(possibly two or all of $u,v,w$ coincide). If $u=v$, then by planarity of $G$ every element of $U$ is adjacent to $b$ and $c$, thus by $3$-connectivity of $G$ as above we recover the construction of $T_m$. Suppose instead that $u,v,w$ are all distinct. The reader may refer to Figure \ref{fig:abcd2}.
\begin{figure}[ht]
	\centering
	\includegraphics[width=6.cm]{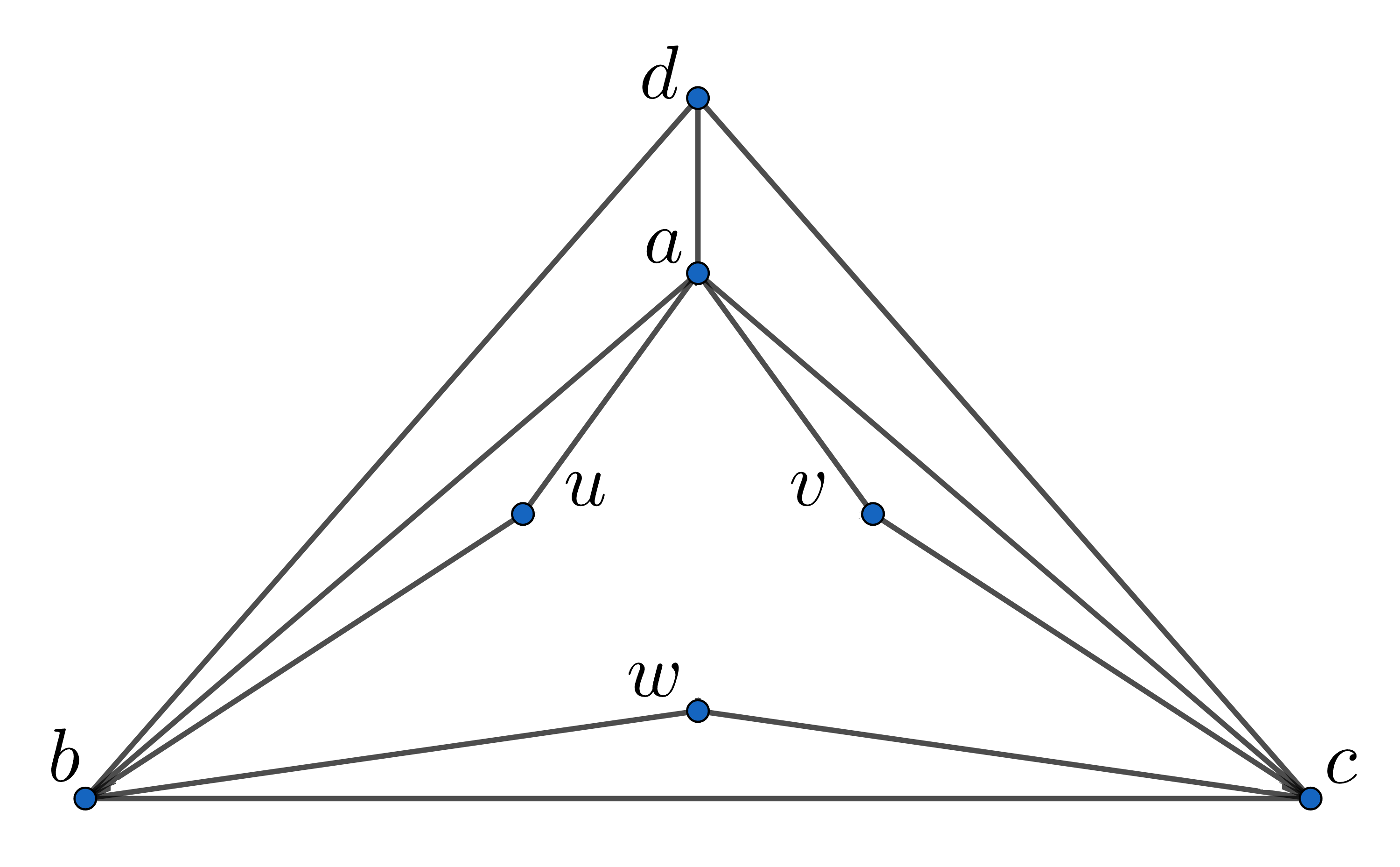}
	\caption{$u,v,w$ are all distinct.}
	\label{fig:abcd2}
\end{figure}

Next, there exists $x\in V(G)$ adjacent to $u,v,w$. If $x\in\{a,b,c\}$ then we have constructed $S_5$ (Figure \ref{fig:s5}). As seen above, there cannot be any further vertices, thus $G\simeq S_5$.

If instead $x\in U$, then as seen above it is adjacent to at least two of $a,b,c$, say $a$ and $b$. Next, $N(v,w,x)\neq\emptyset$, and we cannot have $y\in N(v,w,x)$ with $y\neq c$, else by planarity of $G$, $y\in U$ would be adjacent to at most one of $a,b,c$ (namely $c$). It follows that $N(v,w,x)=\{c\}$. We have recovered $S_7$ (Figure \ref{fig:s7}). As seen above, there cannot be any further vertices, thus $G\simeq S_7$. We have shown that if $G$ is a triangulation verifying $0\not\in A_3(G)$, then $G$ is one of the exceptional graphs $S_5,S_7$ or $T_m$ for some $m\geq 2$.
\end{proof}

Now we will consider $G\in\cq_m$, $m\geq 3$, and show that it verifies \eqref{eq:AQ}.

\paragraph{Case of $\cq_m$.} Let $G\in\cq_m$, $m\geq 3$, and $3\leq n\leq\Delta(G)$. By Lemma \ref{le:no0} and Proposition \ref{prop:no03}, we have $0\in A_n(G)$. Since $n\leq\Delta(G)$, $A_n(G)\neq\{0\}$. By planarity, $A_n(G)\subseteq\{0,1,2\}$. By definition of $\cq_m$, $m$ is the smallest integer such that $|N(v_1,v_2,\dots,v_m)|=2$ holds for every choice of $v_1,v_2,\dots,v_m$ where $N(v_1,v_2,\dots,v_m)\neq\emptyset$. Thereby,
\[A_n(G)=\begin{cases}
	\{0,1,2\}&3\leq n\leq m-1,\\\{0,2\}&m\leq n\leq\Delta(G).
\end{cases}\]
We have verified \eqref{eq:AQ}.

In the other direction, we have the following.
\begin{lemma}
	\label{le:02}
Let $G$ be a planar graph satisfying $A_n(G)=\{0,2\}$ for some $n\geq 3$. Then $G\in\cq_m$ for some $m\leq n$.
\end{lemma}
\begin{proof}
Let $x\in V(G)$ be of degree at least $n$. Since $1\not\in A_n(G)$, there exists an integer $\ell\geq n$ and a vertex $y$ such that
\[N(x,y)=\{w_1,w_2,\dots,w_\ell\}.\]
Our goal is to show that $N(x)=N(y)=N(x,y)$. Suppose by contradiction that there exists $z_1$ such that $z_1x\in E(G)$, $z_1y\not\in E(G)$.

If $\ell\geq 4$, then by planarity $w_1,w_2,\dots,w_{\ell-1}$ have no common neighbours other than $x,y$. Since $1\not\in A_n$, it follows that
\[N(x,w_1,w_2,\dots,w_{\ell-1})=\emptyset,\]
and in particular $xy\not\in E(G)$. Since $1\not\in A_n$, then
\[z_1,w_1,w_2,\dots,w_{\ell-1}\]
have a common neighbour other than $x$, hence $w_1,w_2,\dots,w_{\ell-1}$ have a common neighbour other than $x,y$, contradicting planarity. We deduce that if $\ell\geq 4$ then $\deg(x)=\deg(y)=\ell$.

Now fix instead $\ell=n=3$. Suppose for now that $xy\not\in E(G)$. W.l.o.g.\ we may sketch $z_1$ inside of the quadrangle $x,w_1,y,w_2$. Since $1\not\in A_n$, then $z_1,w_2,w_3$ have a common neighbour other than $x$, and by planarity this common neighbour is $w_1$. Similarly, $z_1,w_1,w_3$ have a common neighbour other than $x$, so that by planarity $w_2z_1,w_2w_3\in E(G)$, as in Figure \ref{fig:z1}.
\begin{figure}[ht]
	\centering
	\includegraphics[width=3.5cm]{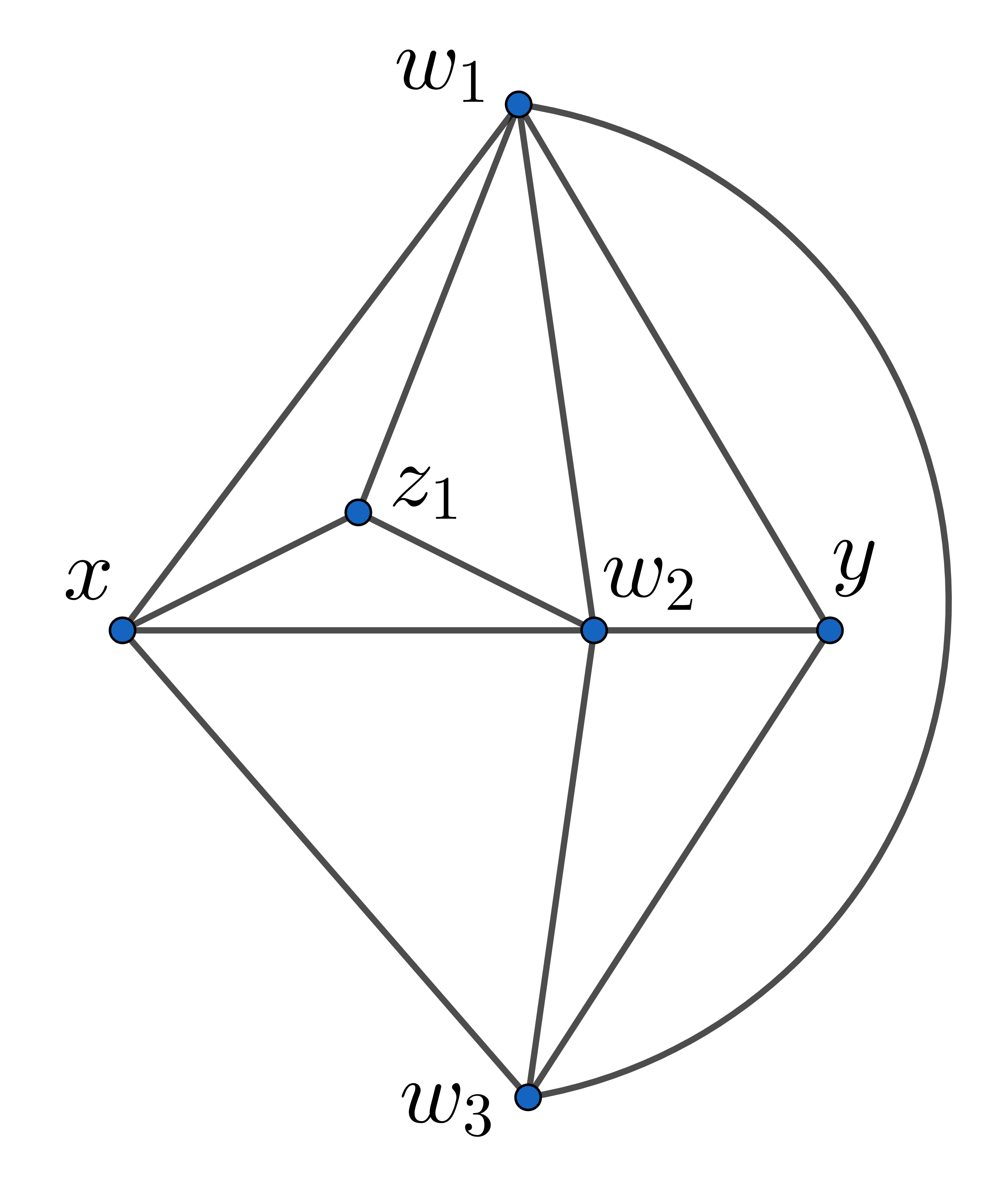}
	\caption{Lemma \ref{le:02}, case of $n=3$.}
	\label{fig:z1}
\end{figure}

Now $z_1,w_1,w_2$ have a common neighbour $z_2$ other than $x$, and by planarity $z_2$ lies inside of the triangle formed by $z_1,w_1,w_2$. Next, $z_2,w_1,w_2$ have a common neighbour $z_3$ other than $z_1$, and by planarity $z_3$ lies inside of the triangle formed by $z_2,w_1,w_2$. Iteratively for each $i\geq 2$, $z_i,w_1,w_2$ have a common neighbour $z_{i+1}$ other than $z_{i-1}$, and by planarity $z_{i+1}$ lies inside of the triangle formed by $z_i,w_1,w_2$. This contradicts the finiteness of $G$ thus $\deg(x)=\deg(y)=\ell$.

Lastly, suppose that $\ell=n=3$ and $z_1=y$. For every $(i,j)\in\{(1,2),(1,3),(2,3)\}$ we have $N(x,w_i,w_j)\neq\{y\}$ thus there exists $u_{i,j}$ such that
\[N(x,w_i,w_j)=\{y,u_{i,j}\}.\]
By planarity, $u_{1,2},u_{1,3},u_{2,3}$ are all distinct. Then $G$ contains a subgraph homeomorphic from $K_{3,3}$ with partition
\[\{\{x,w_1,w_2\},\{y,u_{1,2},u_{1,3}\}\},\]
contradicting planarity.

We have shown that for every $n\geq 3$ if $x$ is a vertex of degree at least $n$, then there exists $y\neq x$ with the same neighbourhood as $x$. That is to say, $G\in\cq_m$ for some $m\leq n$ as claimed.
\end{proof}

\subsection{Concluding the proof}
Let $G\not\in\ce$ (recall \eqref{eq:ce}) be a planar graph and $3\leq n\leq\Delta(G)$ an integer. By Lemma \ref{le:no0} and Proposition \ref{prop:no03}, $0\in A_n(G)$. Since $n\leq\Delta(G)$, $A_n(G)\neq\{0\}$. By Lemma \ref{le:02}, $A_n(G)\neq\{0,2\}$. We deduce that
\[A_n(G)\in\{\{0,1\},\{0,1,2\}\}.\]
Clearly $2\in A_n(G)$ if and only if $G$ contains a subgraph isomorphic to $K_{2,n}$. We have verified \eqref{eq:Agen}, and the proof of Theorem \ref{thm:1} is complete.

\section{Proof of Theorem \ref{thm:3}}
\label{sec:3}

\subsection{$1\in A_2(G)$}

We start by recalling two results from \cite{maffucci2025classification}.
\begin{lemma}[{\cite[Lemma 3.1]{maffucci2025classification}}]
	\label{le:2}
	Let $G$ be a planar graph. If there exists an integer $a\geq 3$ such that $a\in A_2(G)$, then $2\in A_2(G)$.
\end{lemma}
In Lemma \ref{le:2}, the planarity assumption is essential, as illustrated by the complete graphs $K_\ell$, $\ell\geq 5$, that satisfy $A_2(K_\ell)=\{\ell-2\}$, and by the bipartite complete $K_{\ell_1,\ell_2}$, $\ell_2\geq\ell_1\geq 3$, that satisfy $A_2(K_{\ell_1,\ell_2})=\{0,\ell_1,\ell_2\}$.

\begin{lemma}[{\cite[Corollary 3.3]{maffucci2025classification}}]
	\label{le:4cy}
	Let $G$ be a planar graph. Then $G$ contains a $4$-cycle if and only if $2\in A_2(G)$.
\end{lemma}

Next, we establish the following consequence.
\begin{cor}
	\label{cor:01}
	Let $G$ be a planar graph of minimum vertex degree at least $2$. Then
	\[\{1\}\subseteq A_2(G)\subseteq\{0,1\}\]
	if and only if $G$ contains no $4$-cycles.
\end{cor}
\begin{proof}
In one direction, suppose that $\{1\}\subseteq A_2(G)\subseteq\{0,1\}$. Then in particular $2\not\in A_2(G)$, so that by Lemma \ref{le:4cy}, $G$ contains no $4$-cycles.

In the other direction, suppose that $G$ contains no $4$-cycles. By Lemma \ref{le:4cy}, $2\not\in A_2(G)$, thus by Lemma \ref{le:2} we have $A_2(G)\subseteq\{0,1\}$. Since the minimum vertex degree in $G$ is at least $2$, we have $A_2(G)\neq\{0\}$.
\end{proof}

We now study the case $A_2(G)=\{1\}$. Recall the definition of $D_\ell$ \eqref{eq:D}.
\begin{lemma}
	\label{le:1}
	Let $G$ be a planar graph. Then 
	\[A_2(G)=\{1\}\]
	if and only if $G\simeq D_\ell$ for some $\ell\geq 1$.
\end{lemma}
\begin{proof}
Let $A_2(G)=\{1\}$. Given $v\in V(G)$ with $N(v)=\{v_1,v_2,\dots,v_n\}$, then $v,v_1$ have exactly one common neighbour, that we may assume is $v_2$. Similarly, $v,v_3$ have exactly one common neighbour, that we may assume is $v_4$, and continuing in this fashion, we deduce that every $v\in V(G)$ has even degree, and its closed neighbourhood $\Gamma_G[v]$ is isomorphic to $D_{\deg_G(v)/2}$ (cf.\ \cite[Lemma 3.6]{maffucci2025classification}).

By contradiction, let $G\not\simeq D_\ell$ for every $\ell\geq 1$. Then we may assume that $G$ contains the triangles $w_i,v_i,v_{i+1}$ for $1\leq i\leq 3$, as in Figure \ref{fig:vwa}. We note that there can be no more edges between the vertices $v_1,v_2,v_3,v_4,w_1,w_2,w_3$ without obtaining a $4$-cycle, that would contradict Lemma \ref{le:4cy}. Therefore, $N(v_1,v_4)=\{v_5\}$ where $v_5$ is a new vertex. Next, for $i\in\{1,4\}$, $\Gamma_G[v_i]\simeq D_{\ell_i}$ with $\ell_i\geq 2$, thus we draw new vertices $w_4$ adjacent to $v_4,v_5$ and $w_5$ adjacent to $v_5,v_1$, as in Figure \ref{fig:vwb}. Note that $w_4\neq w_5$ otherwise $G$ would contain the $4$-cycle $v_1,v_5,v_4,w_4$.
\begin{figure}[ht]
	\centering
	\begin{subfigure}{0.48\textwidth}
		\centering
		\includegraphics[width=4.cm]{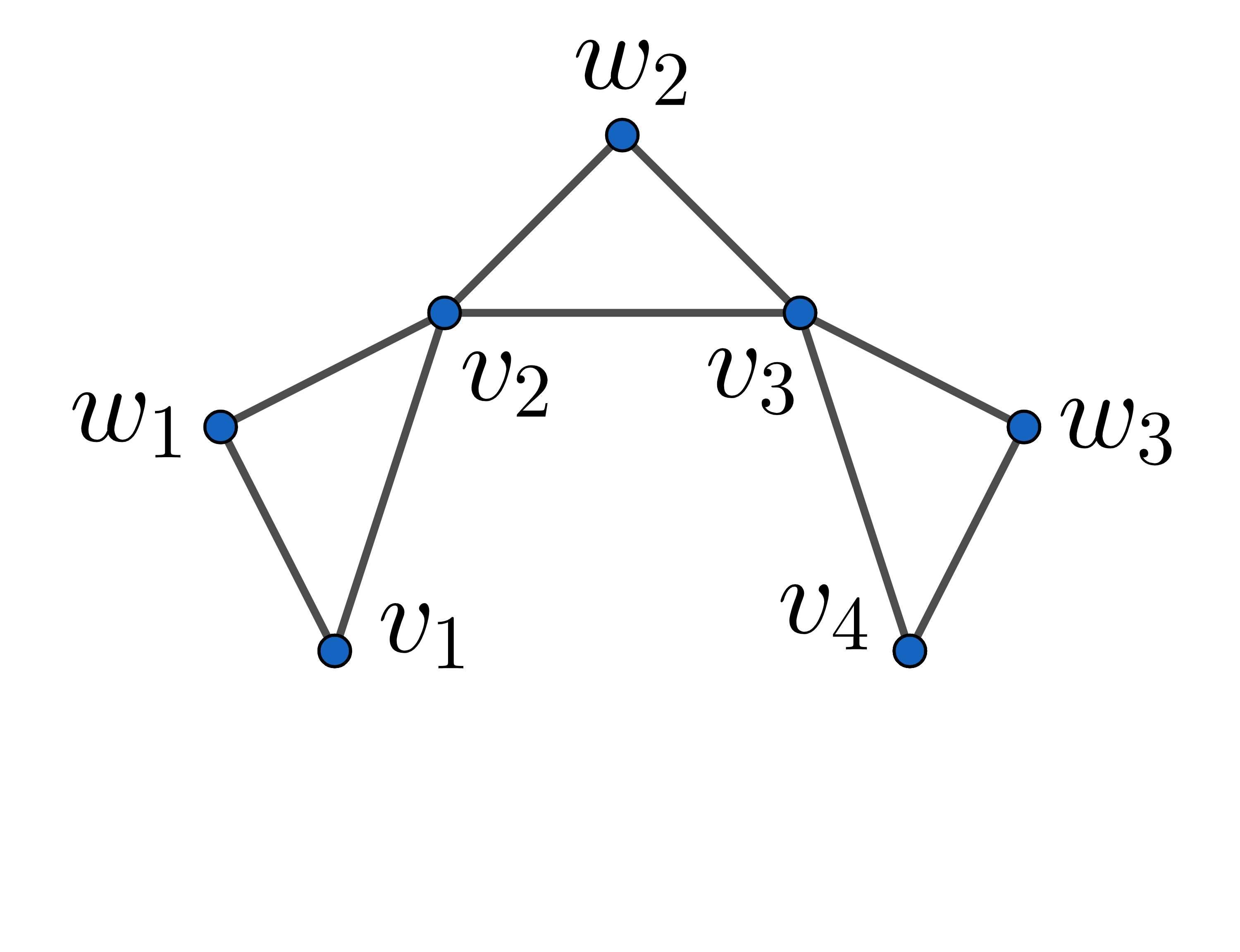}
		\caption{$G$ contains the triangles $w_i,v_i,v_{i+1}$ for $1\leq i\leq 3$.}
		\label{fig:vwa}
	\end{subfigure}
	\hfill
	\begin{subfigure}{0.48\textwidth}
		\centering
		\includegraphics[width=4.cm]{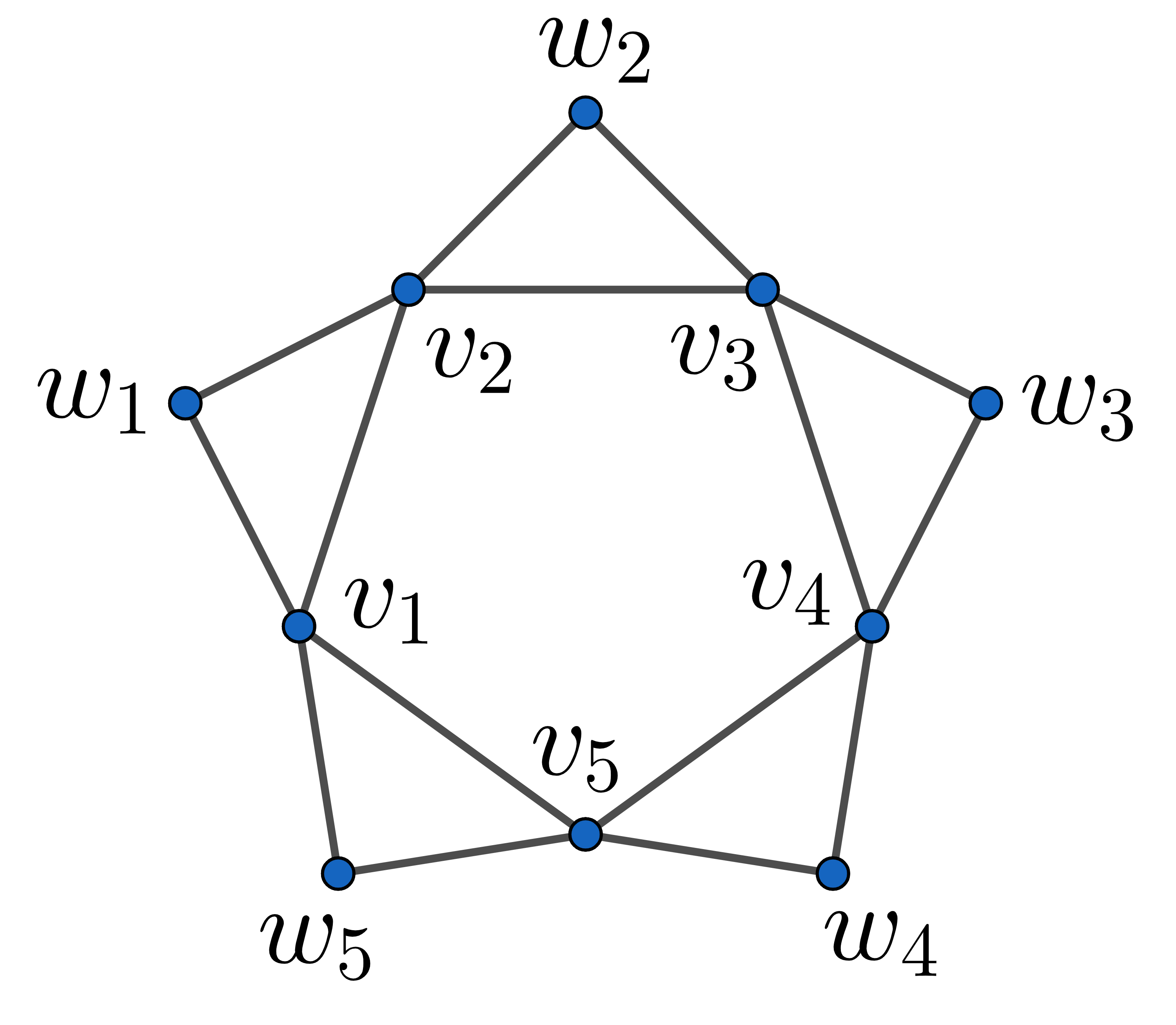}
		\caption{We draw new vertices $w_4$ adjacent to $v_4,v_5$ and $w_5$ adjacent to $v_5,v_1$.}
		\label{fig:vwb}
	\end{subfigure}
	\caption{Lemma \ref{le:1}.}
	\label{fig:vw}
\end{figure}

The vertices $w_1,w_2,w_3,w_4,w_5$ are pairwise non-adjacent, since $G$ contains no $4$-cycles. Thus $N(w_1,w_3)=\{x\}$ and $N(w_2,w_4)=\{y\}$, where $x,y$ are new vertices. If $x\neq y$ we have contradicted planarity, whereas if $x=y$ we have constructed the $4$-cycle $v_3,w_2,x,w_3$.
\end{proof}

Until this moment, we have analysed the case $1\in A_2$, $2\not\in A_2$. On the other hand, when $1,2\in A_2(G)$ we have the following.
\begin{lemma}
	\label{le:vast}
	Let $A'$ be a (possibly empty) set of integers $\geq 3$. Then we may construct a vast class of planar, connected graphs satisfying
	\[A_2(G)=\{1,2\}\cup A',\]
	and we may construct a vast class of planar, connected graphs satisfying
	\[A_2(G)=\{0,1,2\}\cup A'.\]
\end{lemma}
\begin{proof}
We start with an outerplanar graph $G'$ such that
\[A_2(G')=\{0,1\}.\]
Due to Corollary \ref{cor:01}, this is exactly the class of outerplanar graphs $\not\simeq D_\ell$ and containing a $4$-cycle (see Appendix \ref{app:b} for further details). For instance, $G'$ could be any forest. Next, we consider the graph $G=G'+K_1$, where we have added a dominating vertex $x$. Then for every distinct $u,v\in V(G')$ we have
\[|N_G(u,x)|=\deg_{G'}(u)\]
and
\[|N_G(u,v)|=|N_{G'}(u,v)|+1\in\{1,2\}.\]
Therefore,
\[A_2(G)=\{1,2\}\cup\{\deg_{G'}(u) : u\in V(G')\}.\]
%We record that $1\in A_2(G)$ if and only if either $H$ has vertices of degree $1$, or $0\in A_2(H)$ (recall that due to Lemma \ref{le:1}, $0\in A_2(H)$ if and only $H\not\simeq D_\ell$ for any $\ell\geq 1$).
It then suffices to consider $G'$ such that
\[\{\deg_{G'}(u) : u\in V(G'),\ \deg_{G'}(u)\geq 3\}=A'.\]
If $G'$ has isolated vertices, then
\[A_2(G)=\{0,1,2\}\cup A',\]
otherwise
\[A_2(G)=\{1,2\}\cup A'.\]
There are wide classes of solutions seeing as e.g.\ there is a wide class of forests $G'$ satisfying $\{\deg_{G'}(u) : u\in V(G'),\ \deg_{G'}(u)\geq 3\}=A'$ for every given set $A'\neq\emptyset$, whereas if $A'=\emptyset$, we may take for $G'$ any disjoint union of cycles all of length $\neq 4$ (and isolated vertices if we wish that $0\in A_2(G)$).
\end{proof}

\subsection{$1\not\in A_2(G)$}
In the present section, we will complete the proof of Theorem \ref{thm:3}. Let $G$ be a planar, connected graph such that $1\not\in A_2(G)$. Firstly, we note that if $z_0$ is a separating vertex of $G$ and $z_1,z_2\in V(G)$ are in distinct connected components of $G-z_0$, then $N(z_1,z_2)=\{z_0\}$. Therefore, $G$ is $2$-connected. On the other hand, the planar, $3$-connected graphs satisfying $1\not\in A_2$ were classified in \cite[Table 1]{maffucci2025classification}. In the present section it remains to classify the planar graphs of connectivity $2$ satisfying $1\not\in A_2$. Recall the definitions \eqref{eq:tm} and \eqref{eq:Tp}.

\begin{lemma}
	\label{le:exc}
	Let $G$ be a planar graph of connectivity $2$ satisfying $0,1\not\in A_2(G)$. Then either $A_2(G)=\{2,\ell\}$ and $G\simeq T'_\ell$ with $\ell$ even, or $A_2(G)=\{2,3,\ell\}$ and $G\in\ct_\ell$, $G\not\simeq T'_\ell$.
\end{lemma}
\begin{proof}
	Let $\{x,y\}$ be a $2$-cut in $G$, and $w_1,w_2$ neighbours of $x$ in distinct components of $G-x-y$. Since $|N(w_1,w_2)|\neq 1$, the only possibility is $N(w_1,w_2)=\{x,y\}$. Thus we may write
	\[N(x)=N(y)=N(x,y)=\{w_1,w_2,\dots,w_\ell\},\qquad\ell\geq 2,\]
	hence by Lemma \ref{le:2} we have
	\[\{2,\ell\}\subseteq A_2(G).\]
	We may assume that $w_1,w_2,\dots,w_\ell$ appear in this cyclic order around $x$ in a planar immersion of $G$.
	
	%Suppose that $0\not\in A_2(G)$. 
	We claim that
	\begin{equation}
		\label{eq:vset}
		V(G)=\{x,y,w_1,w_2,\dots,w_\ell\}.
	\end{equation}
	By contradiction, let $u$ be a vertex inside of the quadrangle $x,w_i,y,w_{i+1}$ for some $1\leq i\leq\ell-1$. Since $x,y$ are not adjacent to any vertices inside of $x,w_i,y,w_{i+1}$, and since $G-w_i$ and $G-w_{i+1}$ are connected graphs, it follows that there is a $w_iw_{i+1}$-path in $G-x-y$ containing $u$. We deduce that $w_i,w_{i+1}$ are in the same connected component of $G-x-y$. Letting $w_j$ be in a distinct component of $G-x-y$, we obtain $N(u,w_j)=\emptyset$, contradiction.
	
	Hence indeed we have \eqref{eq:vset}. We may assume that $w_1,w_{\ell}$ are in distinct components of $G-x-y$. By planarity,
	\[N(x,w_1)\subseteq\{y,w_{2}\},\]
	hence as $0,1\not\in A_2(G)$, in fact we have $xy,w_1w_{2}\in E(G)$. Similarly, $w_\ell w_{\ell-1}\in E(G)$, and in particular $\ell\geq 4$. Likewise, for each $2\leq i\leq\ell-1$, $w_i$ is adjacent to at least one of $w_{i-1},w_{i+1}$. It follows that $G=H+K_2$, where $H$ is a subgraph of the path $P_\ell$ with no isolated vertices i.e., $G\in\ct_\ell$.
	
	Conversely, let $G\in\ct_\ell$, $\ell\geq 4$, where the vertices along the path $P_\ell$ are labeled $w_1,w_2,\dots,w_\ell$ and $x,y$ are the remaining vertices. Then we have $|N(x,y)|=\ell$. Further, for every $1\leq i<j\leq\ell$, if $j=i+2$ we have
	\[\{x,y\}\subseteq N(w_i,w_{i+2})\subseteq\{x,y,w_{i+1}\},\]
	otherwise we have
	\[N(w_i,w_{j})=\{x,y\}.\]
	We also have for every $2\leq i\leq\ell-1$,
	\[N(x,w_i)\in\{\{y,w_{i-1}\},\{y,w_{i+1}\},\{y,w_{i-1},w_{i+1}\}\}\]
	and the above holds also if we swap $x,y$. It follows that
	\[A_2(G)\in\{\{2,\ell\},\{2,3,\ell\}\},\]
	and moreover $3\not\in A_2(G)$ if and only if $\ell$ is even and
	\[G\simeq \left(\bigcup_{i=1}^{\ell/2}K_2\right)+K_2\simeq T'_\ell.\]
\end{proof}

We now focus on the remaining case $0,2\in A_2$, $1\not\in A_2$. Recall that $S'$ is the graph in Figure \ref{fig:s11}.
\begin{lemma}
	\label{le:023}
	Let $G$ be a planar graph of connectivity $2$ satisfying $A_2(G)\in\{\{0,2\},\{0,2,3\}\}$. Then $G\in\{K_{2,2},K_{2,3},S'\}$.
\end{lemma}
\begin{proof}
	The first argument of Lemma \ref{le:exc} shows that since $G$ is of connectivity $2$ and $1\not\in A_2(G)$, then there exist $\ell\geq 2$ and $x,y,w_1,w_2,\dots,w_\ell\in V(G)$ satisfying
	\[N(x)=N(y)=N(x,y)=\{w_1,w_2,\dots,w_\ell\},\]
	thus
	\[\{2,\ell\}\subseteq A_2(G).\]
	If $\ell=2$, then $w_1,w_2$ are in distinct components of $G-x-y$, so that we may not add any further vertices to $G$ without breaking $2$-connectivity. Hence in this case $G\simeq K_{2,2}$.
	
	Now assume instead that $\ell=3$. If $G-x-y$ has three connected components, then as above we have $G\simeq K_{2,3}$ by $2$-connectivity. The only other possibility is that w.l.o.g.\ $w_1,w_2$ are in the same component of $G-x-y$. Any additional vertex of $G$ is adjacent to both of $w_1,w_2$. Since $w_1,w_2$ already have the two common neighbours $x,y$, then there may be at most one additional vertex. If there is one, then $G\simeq S'$.
\end{proof}

On the other hand, we have the following.
\begin{lemma}
	Let $A'\neq\emptyset,\{3\}$ be a set of integers $\geq 3$. Then we may construct a vast class of planar, connected graphs satisfying 
	\[A_2(G)=\{0,2\}\cup A'.\]
\end{lemma}
\begin{proof}
	We write
	\[A'=\{a_1,a_2,\dots,a_{|A'|}\}\]
	with elements in ascending order. we begin by sketching the graph $K_{2,a_1}$, with edges
	\[u_{0,1}u_{1,i},u_{0,2}u_{1,i}, \qquad 1\leq i\leq a_1.\]
	Then for $2\leq j\leq|A'|$ we perform the transformation
	\[\ft_{a_j} : H\to H+u_{j-1,1}u_{j,i}+u_{j-1,2}u_{j,i}, \qquad 1\leq i\leq a_j-2,\]
	where addition means we are inserting the relevant vertices and edges. Note that at each step we introduce at least two new vertices, since $a_j\geq 4$ for every $j\geq 2$.
	
	In the resulting final graph $G$ we have
	\[|N(u_{j-1,1},u_{j-1,2})|=a_{j},\qquad 1\leq j\leq |A'|\]
	and
	\[N(u_{0,1},u_{1,1})=\emptyset.\]
	Moreover, by construction whenever two vertices $x,y$ have common neighbours, they have either $2$ or $a_j$ of them for some $1\leq j\leq |A'|$. Therefore,
	\[A_2(G)=\{0,2\}\cup A'.\]
	We obtain a vast class of graphs by performing additional transformations $\ft_{a}$ for $a\in A$ such that $a\geq 4$, and by permuting the order of transformations in any way.
\end{proof}

The results of Section \ref{sec:3} have populated Table \ref{t:1}, completing the proof of Theorem \ref{thm:3}.

\appendix
\section{Common neighbours of polyhedra}
\label{app:a}

Recall the definition of the classes of graphs $\ct_m$ \eqref{eq:tm} and $\cq_m$ \eqref{eq:qm}. Clearly a member of $\ct_m$, $m\geq 2$ is a polyhedron if and only if it is $T_m$ of Figure \ref{fig:c1}. To classify polyhedra according to $A_n(G)$ via Theorems \ref{thm:1} and \ref{thm:2}, it remains to determine which elements of $\cq$ are polyhedra. Recall the families of graphs $B_\ell,B'_\ell$ of Figures \ref{fig:c2} and \ref{fig:c3}.

\begin{lemma}
	\label{le:B}
	Let $G$ be a polyhedron satisfying $G\in\cq_m$. Then either $m=3$ and $G$ is the octahedron $B_4$, or $m=5$ and $G\in\{B_\ell,B'_\ell\}$ for some $\ell\geq 5$.
\end{lemma}
\begin{proof}
	Let $G\in\cq_m$ with $m\geq 3$. Take $x\in V(G)$ of degree $\ell\geq m$. By assumption, there exists $y\neq x$ such that
	\[N(x)=N(y)=\{w_1,w_2,\dots,w_\ell\}.\]
	We record that in particular $xy\not\in E(G)$. We may assume that $w_1,w_2,\dots,w_\ell$ appear in this order around $x$ in the planar immersion of $G$. Suppose by contradiction that there exists $u\in V(G)$ inside of the cycle
	\begin{equation}
		\label{eq:face}
		x,w_i,y,w_{i+1}
	\end{equation}
	(where indices $w_i$ are taken modulo $\ell$). Then $G-w_i-w_{i+1}$ is disconnected, contradicting the $3$-connectivity of $G$. Hence $|V(G)|=\ell+2$, and moreover either \eqref{eq:face} is a face or $w_iw_{i+1}\in E(G)$. Now at most one of
	\[x,w_i,y,w_{i+1} \quad 1\leq i\leq\ell\]
	may be a face, otherwise we would have in $G$ two faces with two common non-adjacent vertices. Therefore w.l.o.g.
	\[w_2,w_3,w_4,\dots,w_\ell,w_1\]
	form a path in $G$. It follows that
	\[G\simeq
	\begin{cases}
		B_\ell & w_1w_2\in E(G),\\
		B'_\ell & w_1w_2\not\in E(G).
	\end{cases}
	\]
	We have shown that if a polyhedron $G$ satisfies $G\in\cq_m$ for some $m\geq 3$, then it is one of $B_\ell,B'_\ell$ for some $\ell\geq m$. In the other direction, one checks that $B_4\in\cq_3$ and $B_\ell,B'_\ell\in\cq_5$ for every $\ell\geq 5$, whereas $B_3,B'_3,B'_4\not\in\cq$.
\end{proof}

We now have the following consequence of Theorem \ref{thm:2}. Recall the definition of $\fa$ \eqref{eq:fa}, $L$ \eqref{eq:L} and of the polyhedra in Figures \ref{fig:c} and \ref{fig:s}.
\begin{prop}
	Let $G$ be a polyhedron. Then we have
	\[
		\fa(G)=\{\{0,1,2\},\{0,1,2\},\dots,\underbrace{\{0,1,2\}}_{A_L},\{0,1\},\{0,1\},\dots,\{0,1\},\underbrace{\{0\}}_{A_{\Delta+1}},\{0\},\dots,\underbrace{\{0\}}_{A_p},\emptyset,\emptyset,\dots,\emptyset,\dots\}
	\]
	except for the cases
	\begin{align*}
		\fa(T_m)&=\{\underbrace{\{1,2\}}_{A_{3}},\{0,1,2\},\{0,1,2\},\dots,\{0,1,2\},\underbrace{\{0,1\}}_{A_{m+1}},\{0\},\emptyset,\emptyset,\dots,\emptyset,\dots\},\qquad m\geq 3,\\
		\fa(B_\ell)=\fa(B'_\ell)&=\{\{0,1,2\},\{0,1,2\},\{0,2\},\{0,2\},\dots,\underbrace{\{0,2\}}_{A_\ell},\{0\},\{0\},\emptyset,\emptyset,\dots,\emptyset,\dots\},\qquad\ell\geq 5,
		\\
		\fa(K_4)&=\{\{1\},\{0\},\emptyset,\emptyset,\dots,\emptyset,\dots\},\\
		%\fa(T_3)&=\{\{1,2\},\{0,1\},\{0\},\emptyset,\emptyset,\dots,\emptyset,\dots\},\notag\\
		\fa(B_4)&=\{\{0,2\},\{0,2\},\{0\},\{0\},\emptyset,\emptyset,\dots,\emptyset,\dots\},\\
		%\fa(B'_4)&=\{\{0,1,2\},\{0,1,2\},\{0\},\{0\},\emptyset,\emptyset,\dots,\emptyset,\dots\},\notag\\
		\fa(S_5)&=\{\{1,2\},\{0,1,2\},\{0,1\},\{0,1\},\{0\},\emptyset,\emptyset,\dots,\emptyset,\dots\},\\
		\fa(S_7)&=\{\{1,2\},\{0,1,2\},\{0,1\},\{0,1\},\{0\},\{0\},\emptyset,\emptyset,\dots,\emptyset,\dots\}.
	\end{align*}
\end{prop}

\section{Common neighbours of outerplanar graphs}
\label{app:b}

The class of outerplanar graphs is a widely studied subclass of planar graphs, that was useful in the constructions of Lemma \ref{le:vast}. In the present section we classify it according to $A_n(G)$, for $n\geq 3$ and for $n=2$.

We note that if $G\in\ct\setminus\{K_4\}$ or if $G\in\cq$, then $G$ contains a subgraph isomorphic to $K_{2,3}$. Therefore, if $G\in\ct\cup\cq$ then $G$ is not outerplanar \cite[Theorem 11.10]{harary}. Then we have the following consequence of Theorem \ref{thm:2}.
\begin{cor}
	Let $G$ be an outerplanar graph. Then
	\begin{equation}
		\fa(G)=\{\{0,1,2\},\{0,1,2\},\dots,\underbrace{\{0,1,2\}}_{A_L},\{0,1\},\{0,1\},\dots,\{0,1\},\underbrace{\{0\}}_{A_{\Delta+1}},\{0\},\dots,\underbrace{\{0\}}_{A_p},\emptyset,\emptyset,\dots,\emptyset,\dots\}.
	\end{equation}
\end{cor}

We now focus on $A_2(G)$. Let
\[\calr_\ell:=\{H+K_1 : H\text{ is a subgraph of }P_\ell\text{ with no isolated vertices}\}\]
and
\[\calr:=\bigcup_{\ell\geq 3}\calr_\ell.\]
Note that if $R=H+K_1\in\calr_\ell$, then $R+K_1\simeq H+K_2\in T_{\ell}$ (Figure \ref{fig:c1}).

\begin{cor}
The classification of outerplanar, connected graphs $G$ according to $A_2(G)$ may be found in Table \ref{t:2}.
\begin{table}[ht]
	\centering
	$\begin{array}{|c|c|}
		\hline A_2(G)&G\\
		\hline \emptyset&K_1\\
		\hline \{0\}&K_2\\
		\hline \{1\}&D_\ell,\ \ell\geq 1\\
		\hline \{0,1\}&\text{no $4$-cycles},\ G\not\simeq D_\ell,K_1,K_2\\
		\hline \{0,2\}&\text{square}\\
		\hline \{1,2\}&\calr\\
		\hline \{0,1,2\}&\text{otherwise}\\
		\hline
	\end{array}$
	\caption{Classification of outerplanar, connected graphs $G$ according to $A_2(G)$.}
	\label{t:2}
\end{table}

\end{cor}
\begin{proof}
If $G$ is outerplanar, then it does not contain a subgraph isomorphic to $K_{2,3}$ \cite[Theorem 11.10]{harary}. Hence for every distinct $u,v\in V(H)$ we have $|N(u,v)|\leq 2$, thus
\[A_2(G)\subseteq\{0,1,2\}.\]
Combining with the classification of planar, connected graphs in Table \ref{t:1}, we may thus populate all except the last two rows in Table \ref{t:2}.

Let $G$ be an outerplanar graph satisfying $A_2(G)=\{1,2\}$. The connectivity of $G$ may be either $1$ or $2$ \cite[Corollary 11.9(a)]{harary}. If it is $1$, call $v$ a separating vertex of $G$. Since $0\not\in A_2(G)$, $v$ is a dominating vertex. We denote by
\[w_1,w_2,\dots,w_{\ell},\qquad \ell=|V(G)|-1\]
the neighbours of $v$.

For every $1\leq i\leq\ell$ one has
$N(w_i,v)=\deg_{G-v}(w_i)$, so that $\deg_{G-v}(w_i)\in\{1,2\}$. Further, $G-v$ cannot contain cycles, otherwise $G$ would contain a subgraph homeomorphic to $K_4$, thus $G$ would not be outerplanar \cite[Theorem 11.10]{harary}. It follows that $G-v$ is a subgraph of the path $P_\ell$ with no isolated vertices i.e., $G\simeq\cal{R}_\ell$.

If instead $G$ is of connectivity $2$, then it is Hamiltonian \cite[Lemma 3.3]{maffucci2025regularity}, and we may start by writing the Hamiltonian cycle
\[C: w_0,w_1,w_2,\dots,w_\ell.\]
By outerplanarity, all remaining edges will be drawn inside of $C$. Note that $G$ itself cannot be a cycle, as the triangle $D_1$ satisfies $A_2=\{1\}$, and every other cycle satisfies $0\in A_2$. Hence w.l.o.g.\ we may add the edge $w_0w_i$ for some $2\leq i\leq\ell-1$. Letting $1\leq j\leq i-1$ and $i+1\leq k\leq\ell$, since
\[N(w_j,w_k)\neq\emptyset,\]
we deduce that $w_j,w_k$ are both adjacent to at least one of $w_0,w_i$. By outerplanarity, it follows that w.l.o.g.\ $w_0w_j\in E(G)$ for every $1\leq j\leq\ell$. Thus
\[G\simeq P_\ell+K_1\in\calr_\ell,\]
as claimed.

\end{proof}

\section{Classifying planar graphs according to $A_1(G)$}
\label{app:c}

It is well-known that every planar graph contains at least one vertex of degree at most $5$. We will now show that when this necessary condition is fulfilled, we may construct a vast class of planar, connected graphs $G$ for each feasible $A_1(G)$.
\begin{prop}
	\label{prop:a1}
Let $A$ be a set of integers such that $\min(A)=a$, $1\leq a\leq 5$. Then we may construct a vast class of planar, connected graphs $G$ satisfying $A_1(G)=A$.
\end{prop}
\begin{proof}
Given any set of positive integers $A$, we may construct a vast class of $(p,q)$-trees $G'$ where all vertex degrees belong to the set $A\cup\{1\}$. For instance, we may take a caterpillar with central vertices of degrees the elements of $A$ in any order, possibly with multiplicities. If $a=1$ we are done. Otherwise, letting
\[V_1=\{v\in V(G') : \deg(v)=1\}, \qquad p_1=|V_1|,\]
we wish to perform a graph transformation as follows. For $a=2$, assuming that $2\mid p_1$ we simply add a perfect matching between the elements of $V_1$. 

For $3\leq a\leq 5$, assuming that $p_1\geq 3$ we start by adding a cycle $C$ containing all elements of $V_1$. In case $a=3$ we are done. In case $a=4$, assuming that $4\mid p_1$ we partition the vertices of $C$ into subsets of cardinality $4$, with each subset containing consecutive vertices along $C$. For each subset, as in Figure \ref{fig:u4bis} we then add a vertex (the central one in the illustration) and edges connecting it to the four elements of the subset (the other four vertices in the illustration).

Similarly in case $a=5$,  assuming that $10\mid p_1$ we partition the vertices of $C$ into subsets of cardinality $10$, with each subset containing consecutive vertices along $C$. For each subset, as in Figure \ref{fig:u5bis} we then add twelve vertices (the central ones in the illustration) and edges as in the figure.

Since $G'$ is a tree, each of the described transformations preserves planarity. In each case they produce a graph $G$ with the desired minimum vertex degree $a$.
\begin{figure}[ht]
	\centering
	\begin{subfigure}{0.48\textwidth}
		\centering
		\includegraphics[width=2.25cm]{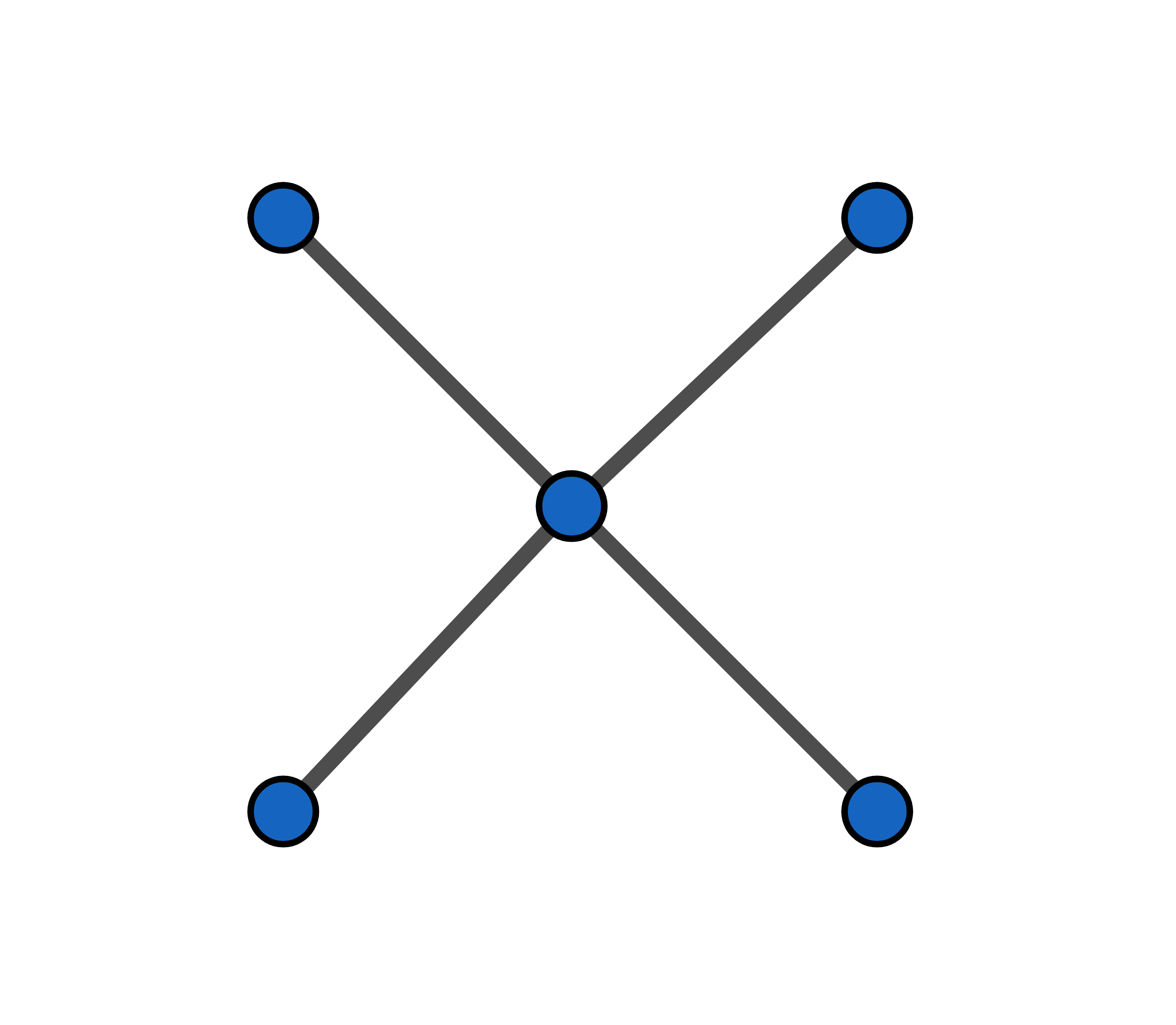}
		\caption{$a=4$.}
		\label{fig:u4bis}
	\end{subfigure}
	\hfill
	\begin{subfigure}{0.48\textwidth}
		\centering
		\includegraphics[width=3.5cm]{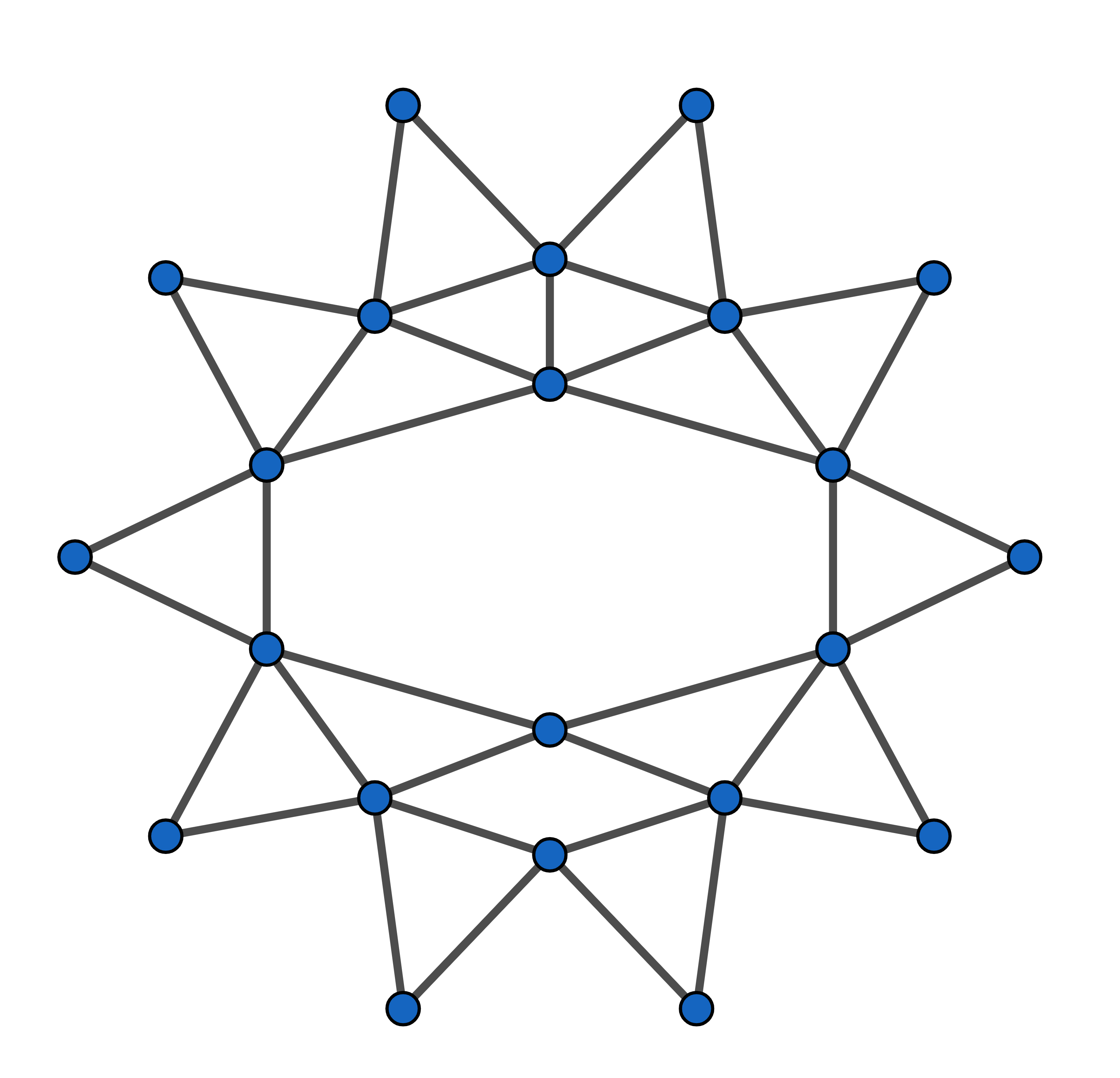}
		\caption{$a=5$.}
		\label{fig:u5bis}
	\end{subfigure}
	\caption{Proposition \ref{prop:a1}.}
	\label{fig:u}
\end{figure}

It remains to show that for every set of positive integers $A$, we may construct starting from $G'$ a tree $G''$ where all vertex degrees belong to the set $A\cup\{1\}$ and such that $p_1(G'')$ has the desired congruence property in the cases $a\in\{2,4,5\}$. Then starting from $G''$ and proceeding as above, one obtains a planar, connected graph $G$ with $A_1(G)=A$.

The degree sequence of $G'$ may be written as
\[\sigma': a_1^{m_1},a_2^{m_2},\dots,a_\ell^{m_\ell},1^{p_1},\]
where the exponents are positive integers representing quantities of vertices of the given degree, and $a_1,a_2,\dots,a_\ell$ are the elements of $A$. Since $q=p-1$ as $G'$ is a tree, and by the handshaking lemma,
\[p_1+\sum_{i=1}^{\ell}a_im_i=2q=-2+2p=-2+2\left(p_1+\sum_{i=1}^{\ell}m_i\right),\]
thus
\begin{equation}
	\label{eq:p1}
p_1=2+\sum_{i=1}^{\ell}(a_i-2)m_i.
\end{equation}
If $a=5$, we may add any number of vertices of degree $5$ to $G'$ (for instance if $G'$ is a caterpillar, by prolonging the central path). Due to \eqref{eq:p1}, adding each of these vertices increases $p_1$ by $5-2=3$, and since $3,10$ are coprime, we are able to construct a tree $G''$ such that $10\mid p_1(G'')$ and $A_1(G'')=A_1(G')=A$.

If $a\in\{2,4\}$, then we take for $G'$ an even number of vertices of odd degree (possibly none, if $A$ contains only even integers), so that $p_1$ is even. If $a=4$ and $4\mid p_1(G')$, or if $a=2$, we simply take $G''=G'$ and we are done. Lastly, if $a=4$ and $p_1(G')\equiv 2\pmod 4$, we add to $G'$ a vertex of degree $4$ so that $p_1$ is increased by $4-2=2$, and the proof is complete.
\end{proof}

Recall that polyhedral graphs have minimal vertex degree at least $3$, and outerplanar graphs have minimal vertex degree at most $2$. On the other hand, we note that for $3\leq a\leq 5$, the graphs $G$ constructed in Proposition \ref{prop:a1} are all polyhedra, and for $1\leq a\leq 2$, the graphs $G$ constructed in Proposition \ref{prop:a1} are all outerplanar. We deduce the following.
\begin{cor}
	\label{cor:a1}
	Let $A$ be a set of integers such that $\min(A)=a$, $3\leq a\leq 5$. Then we may construct a vast class of polyhedra $G$ satisfying $A_1(G)=A$.
\end{cor}

\begin{cor}
	Let $A$ be a set of integers such that $\min(A)=a$, $1\leq a\leq 2$. Then we may construct a vast class of outerplanar, connected graphs $G$ satisfying $A_1(G)=A$.
\end{cor}

%\begin{figure}[ht]
	%\centering
	%\includegraphics[width=\textwidth]{}
	%\caption{}
	%\label{}
%\end{figure}

\paragraph{Future directions.}
The literature of planar graphical sequences has had recent important contributions \cite{bar2024sparse,bar2025approximate}. In Section \ref{sec:disc} we have proposed $n$-degree sequences as a generalisation of degree sequences. It would be interesting to understand facts relating to $n$-degree sequences of planar graphs, such as graphicity, unigraphicity, and constructing realisations for a given $n$-degree sequence.
\\
Another possible direction is the study of the sets $A_n(G)$ for other relevant classes of graphs.

\paragraph{Acknowledgements.}
Riccardo W. Maffucci was partially supported by Programme for Young Researchers `Rita Levi Montalcini' PGR21DPCWZ \textit{Discrete and Probabilistic Methods in Mathematics with Applications}, awarded to Riccardo W. Maffucci.

%\clearpage
%\addcontentsline{toc}{section}{References}
\bibliographystyle{abbrv}
\bibliography{bibgra}
\end{document}